\documentclass{amsart}
\usepackage{todonotes}
\usepackage{mathpazo}
\usepackage{enumerate}
\usepackage{subcaption}
\usepackage{ifthen}
\usepackage{tikz,pgf,pgfmath}
	\tikzstyle{fnode}=[fill=black,draw=black,circle,scale=\s]
	\pgfdeclarelayer{background}
	\pgfdeclarelayer{middle}
	\pgfsetlayers{background,middle,main}
\usepackage[pagebackref,colorlinks=true,
	urlcolor=blue!20!black,
	citecolor=green!50!black]{hyperref}
\usepackage{lscape}
\usepackage{amssymb}
\usepackage[nobysame]{amsrefs}
\usepackage{shuffle}
\makeatletter
\let\orgdescriptionlabel\descriptionlabel
\renewcommand*{\descriptionlabel}[1]{%
  \let\orglabel\label
  \let\label\@gobble
  \phantomsection
  \protected@edef\@currentlabel{#1\unskip}%
  \let\label\orglabel
  \orgdescriptionlabel{(#1)}%
}
\makeatother
\newtheorem{theorem}{Theorem}[section]
\newtheorem{proposition}[theorem]{Proposition}
\newtheorem{conjecture}[theorem]{Conjecture}
\newtheorem{lemma}[theorem]{Lemma}
\newtheorem{corollary}[theorem]{Corollary}
\theoremstyle{remark}

\newtheorem{example}[theorem]{Example}
\newtheorem{remark}[theorem]{Remark}

\newcommand{\defn}[1]{{\color{green!50!black}\emph{#1}}}
\newcommand{\defs}{\stackrel{\mathrm{def}}{=}}
\newcommand{\ie}{i.e.\;}
\renewcommand{\th}{^{\mathsf{th}}}
\newcommand{\st}{^{\mathsf{st}}}

\newcommand{\Tri}{\mathsf{Tri}}
\newcommand{\Hoch}{\textsf{\textbf{Hoch}}}
\newcommand{\Free}{\mathsf{Free}}
\newcommand{\comp}{\leq_{\mathsf{comp}}}
\newcommand{\compless}{\lessdot_{\mathsf{comp}}}
\newcommand{\perspective}{\doublebarwedge}
\newcommand{\CLO}{\textsf{\textbf{CLO}}}
\newcommand{\CJC}{\mathsf{CJC}}
\newcommand{\link}{\mathsf{link}}
\newcommand{\del}{\mathsf{del}}
\newcommand{\Bool}{\textsf{\textbf{Bool}}}
\newcommand{\Tamari}{\textsf{\textbf{Tam}}}
\newcommand{\Can}{\mathsf{Can}}
\newcommand{\Atom}{\mathsf{Atom}}
\newcommand{\Cov}{\mathsf{Cov}}
\newcommand{\Core}{\mathsf{C}}
\newcommand{\CP}{\mathsf{CP}}

\renewcommand{\neg}{\mathsf{neg}}
\newcommand{\pos}{\mathsf{pos}}
\newcommand{\Lattice}{\mathbf{L}}
\newcommand{\Galois}{\mathsf{Galois}}
\newcommand{\Poset}{\mathbf{P}}
\newcommand{\least}{\hat{0}}
\newcommand{\grtst}{\hat{1}}
\newcommand{\rk}{\mathsf{rk}}
\newcommand{\ZetaPol}{\mathcal{Z}}
\newcommand{\AC}{\mathsf{Anti}}

\newcommand{\len}{\mathsf{len}}
\newcommand{\MI}{\mathsf{MeetIrr}}
\newcommand{\JI}{\mathsf{JoinIrr}}
\newcommand{\JIPoset}{\textsf{\textbf{JoinIrr}}}
\newcommand{\Covers}{\mathsf{Covers}}
\newcommand{\jsdlabeling}{\lambda_{\mathsf{jsd}}}
\newcommand{\Shuffle}{\mathsf{Shuf}}
\newcommand{\ShufflePoset}{\textsf{\textbf{Shuf}}}
\newcommand{\ab}{\mathbf{a}}
\newcommand{\bb}{\mathbf{b}}
\newcommand{\wb}{\mathbf{w}}
\newcommand{\one}{\mathbb{1}}
\newcommand{\nil}{\varepsilon}
\newcommand{\afr}{\mathfrak{a}}
\newcommand{\bfr}{\mathfrak{b}}
\newcommand{\jfr}{\mathfrak{j}}
\newcommand{\mfr}{\mathfrak{m}}
\newcommand{\ofr}{\mathfrak{o}}
\newcommand{\tfr}{\mathfrak{t}}
\newcommand{\ufr}{\mathfrak{u}}
\newcommand{\vfr}{\mathfrak{v}}
\newcommand{\wfr}{\mathfrak{w}}
\newcommand{\Jb}{\mathbf{J}}
\newcommand{\idls}[3]{
	\begin{tikzpicture}[scale=#3]\small
		\def\d{.5};
		\def\s{.8*#3};
		\draw(4.6*\d,4.6*\d) node{};
		\draw(1*\d,1*\d) node[draw,circle,scale=\s](a1){};
		\draw(1*\d,2*\d) node[draw,circle,scale=\s](a2){};
		\draw(1*\d,3*\d) node[draw,circle,scale=\s](a3){};
		\draw(1*\d,4*\d) node[draw,circle,scale=\s](a4){};
		\draw(2*\d,1*\d) node[draw,circle,scale=\s](b2){};
		\draw(3*\d,1*\d) node[draw,circle,scale=\s](b3){};
		\draw(4*\d,1*\d) node[draw,circle,scale=\s](b4){};
		\draw(a1) -- (a2) -- (a3) -- (a4);
		\draw(b2) -- (a2);
		\foreach \a/\b in {#1}{
			\draw(\a*\d,\b*\d) node[fill,circle,scale=\s]{};
			\ifthenelse{\equal{\b}{1}}
			{\draw(\a*\d,\b*\d) node[draw=red,circle,scale=1.75*\s]{};}
			{}
		}
		\draw(3.5*\d,3.5*\d) node{#2};
	\end{tikzpicture}
}
\title{Hochschild lattices and shuffle lattices}
\author{Henri M{\"u}hle}
\address{Technische Universit{\"a}t Dresden, Institut f{\"u}r Algebra, Zellescher Weg 12--14, 01069 Dresden, Germany.}
\email{henri.muehle@tu-dresden.de}

\begin{document}

\allowdisplaybreaks

\begin{abstract}
	In his study of a Hochschild complex arising in connection with the free loop fibration, S.~Saneblidze defined the freehedron, a certain polytope constructed via a truncation process from the hypercube.  It was recently conjectured by F.~Chapoton and proven by C.~Combe that a certain orientation of the $1$-skeleton of the freehedron carries a lattice structure.  The resulting lattice was dubbed the Hochschild lattice and is congruence uniform and extremal.  These properties allow for the definition of three associated structures: the Galois graph, the canonical join complex and the core label order.  In this article, we study and characterize these structures.  We exhibit an isomorphism from the core label order of the Hochschild lattice to a particular shuffle lattice of C.~Greene.  We also uncover an enumerative connection between the core label order of the Hochschild lattice, a certain order extension of its poset of irreducibles and the freehedron.  These connections nicely parallel the situation surrounding the better-known Tamari lattices, noncrossing partition lattices and associahedra.
\end{abstract}

\maketitle

\section{Introduction}
	\label{sec:introduction}
In \cite{saneblidze09bitwisted}, S.~Saneblidze introduced the \defn{freehedron} $\Free(n)$, an $n$-dimensional polytope obtained from the $n$-dimensional hypercube via a certain truncation process.  

In \cite{chapoton20some}, F.~Chapoton defined a new partial order on the set of Dyck paths of semilength $n$; the \defn{dexter order}.  He observed that the Dyck paths in a certain interval of this poset are in bijection with the vertices of the freehedron.  This bijection encodes the vertices of $\Free(n)$ as certain $n$-tuples with entries in $\{0,1,2\}$.  Chapoton conjectured that the relevant interval in the dexter order is actually isomorphic to the orientation of the $1$-skeleton of $\Free(n)$ induced by the componentwise order of these $n$-tuples.

This conjecture was settled by C.~Combe in \cite{combe20geometric}.  In fact, she showed that this orientation of the $1$-skeleton of $\Free(n)$ constitutes the poset diagram of a congruence-uniform, extremal lattice; the \defn{Hochschild lattice}\footnote{The terminology stems from the fact that the freehedron arises in the study of a Hochschild complex arising in the context of the free loop fibration.} $\Hoch(n)$.  These are two intriguing combinatorial lattice properties: \defn{extremal} means that the number of join- and meet-irreducibles of $\Hoch(n)$ equals the length of $\Hoch(n)$, and \defn{congruence uniform} means that $\Hoch(n)$ can be constructed from the singleton lattice by a sequence of interval doublings.  See Section~\ref{sec:lattices} for the precise definitions.

Following \cite{markowsky92primes}, any extremal lattice is uniquely determined by a certain directed graph---the \defn{Galois graph}---much like a distributive lattice is determined by its poset of join-irreducibles.  Our first main result characterizes the Galois graph of $\Hoch(n)$.

\begin{theorem}\label{thm:hochschild_galois_graph}
	For $n>0$, the Galois graph of $\Hoch(n)$ is isomorphic to the directed graph $(V,E)$ with $V=\bigl\{(1,1),(1,2),\ldots,(1,n),(2,2),(2,3),\ldots,(2,n)\bigr\}$ which has an edge $(s,t)\to(s',t')$ if and only if $(s,t)\neq (s',t')$ and 
	\begin{itemize}
		\item either $s=2$, $s'=1$ and $t=t'$,
		\item or $s=s'=1$ and $t>t'$.
	\end{itemize}
\end{theorem}

Any element of a congruence-uniform lattice admits a canonical representation as a join of join-irreducible elements.  Such a \defn{canonical join representation} can be described neatly by an edge-labeling determined by a perspectivity relation, see Section~\ref{sec:join_semidistributive_lattices}.  The set of canonical join representations is closed under passing to subsets, and therefore forms a simplicial complex; the \defn{canonical join complex}~\cite{reading15noncrossing}*{Proposition~2.2}.  See \cite{barnard19canonical} for a general study of canonical join complexes.  Our second main result establishes that the canonical join complex of $\Hoch(n)$ is vertex decomposable, which implies that this complex is shellable and Cohen-Macaulay.

\begin{theorem}\label{thm:hochschild_vertex_decomposable}
	For $n>0$, the canonical join complex of $\Hoch(n)$ is vertex decomposable.
\end{theorem}

The previously mentioned edge-labeling of a congruence-uniform lattice (the one that determines the canonical join representations) gives rise to an alternate partial order on the elements of this lattice.  Informally, with each lattice element we associate the set of edge-labels appearing in a particular interval, and order these sets by inclusion.  See Section~\ref{sec:core_label_order} for the precise definitions.  The resulting order---the \defn{core label order}---was first considered in the context of posets of regions of hyperplane arrangements~\cite{reading11noncrossing} and was later studied in a lattice-theoretic setting in \cites{muehle19the,muehle21distributive}.  We prove that the core label order of $\Hoch(n)$ is a lattice.  In fact, we prove much more than that: we show that the core label order of $\Hoch(n)$ is isomorphic to the shuffle lattice $\ShufflePoset(n-1,1)$ studied by C.~Greene in \cite{greene88posets}.  See Section~\ref{sec:triwords_shuffles} for the exact definitions.

\begin{theorem}\label{thm:hochschild_clo_shuffleposet}
	For $n>0$, the core label order of $\Hoch(n)$ is isomorphic to the shuffle lattice $\ShufflePoset(n-1,1)$.  
\end{theorem}

We end this article with an enumerative observation.  Building on \cite{greene88posets}, we compute the \defn{$M$-triangle} of the core label order of $\Hoch(n)$, a refined variant of the (dual) characteristic polynomial of this lattice.  This polynomial behaves nicely under certain variable substitutions.  More precisely, certain invertible transformations of the $M$-triangle yield two other polynomials---the \defn{$F$-} and the \defn{$H$-triangle}---with nonnegative integer coefficients.  We provide combinatorial realizations of these polynomials in terms of refined enumerations of canonical join representations in $\Hoch(n)$.  Moreover, we provide a combinatorial explanation of the $F$-triangle as a certain face-generating function of the freehedron, and we interpret the $H$-triangle as a generating function of antichains in a particular order extension of the poset of irreducibles of the Hochschild lattice.  See Section~\ref{sec:chapoton_triangles} for the details. 

We wish to emphasize that the results of this paper nicely parallel known phenomena occurring around the Tamari lattice $\Tamari(n)$.  This is a certain lattice defined by a rotation operation on the set of full binary trees with $n$ internal nodes~\cite{tamari51monoides}.  The poset diagram of $\Tamari(n)$ is isomorphic to a particular orientation of the $1$-skeleton of the $n$-dimensional associahedron~\cite{stasheff63homotopy}, a polytope that arises by a certain truncation process from the $n$-dimensional hypercube, too.  The lattice $\Tamari(n)$ is congruence-uniform and extremal~\cites{geyer94on,markowsky92primes}, and its core label order is isomorphic to the lattice of noncrossing set partitions of an $n$-element set~\cite{reading11noncrossing}; see also \cite{kreweras72sur}.  The Galois graph of $\Tamari(n)$ was computed in \cite{markowsky92primes}, and it was shown in \cite{barnard20the} that the canonical join complex of $\Tamari(n)$ is vertex decomposable.  The $M$-triangle of the noncrossing partition lattice was computed in \cite{athanasiadis07on} following a conjectural description in \cites{chapoton04enumerative,chapoton06sur}, where the corresponding $F$- and $H$-triangles were defined, too.  Since then, the $F$-triangle has been realized as a refined face count of the (dual) associahedron and the $H$-triangle has been explained combinatorially in terms of antichains in a certain order extension of the poset of irreducibles of $\Tamari(n)$.  See also Section~\ref{sec:h_triangle_irreducibles} and Figure~\ref{fig:polytopes_lattices}.

\medskip

This article is organized as follows.  In Section~\ref{sec:preliminaries}, we recall the necessary order- and lattice-theoretic notions and formally define the Hochschild lattice $\Hoch(n)$.  The Galois graph of $\Hoch(n)$ is defined and computed in Section~\ref{sec:hochschild_galois_graph}.  In Section~\ref{sec:hochschild_join_complex}, we define the canonical join complex, we compute the canonical join representations in $\Hoch(n)$ and prove that the canonical join complex of $\Hoch(n)$ is vertex decomposable.  The core label order is defined in Section~\ref{sec:hochschild_core_label_order}, in which we also prove the connection between the core label order of $\Hoch(n)$ and a particular shuffle lattice.  In Section~\ref{sec:chapoton_triangles}, we compute and explain the $F$-, $H$- and $M$-triangles associated with $\Hoch(n)$.  We end this article with a list of open questions in Section~\ref{sec:open_questions}.

\section{Preliminaries}
	\label{sec:preliminaries}
\subsection{Posets}
	\label{sec:posets}
Let $\Poset=(P,\leq)$ be a partially ordered set (\defn{poset} for short).  In this article we consider only posets $\Poset$ whose \defn{ground set} $P$ is finite.

An element $a\in P$ is \defn{minimal} (resp. \defn{maximal}) if for every $b\in P$ with $b\leq a$ (resp. $a\leq b$) it follows that $b=a$.  A poset is \defn{bounded} if it has a unique minimal and a unique maximal element; usually denoted by $\least$ and $\grtst$, respectively.

For $a,b\in P$ with $a\leq b$, the set $[a,b]\defs\{c\in P\mid a\leq c\leq b\}$ is an \defn{interval} of $\Poset$.  If the interval $[a,b]$ has cardinality two, then the pair $(a,b)$ is a \defn{cover relation} of $\Poset$.  We usually write $a\lessdot b$ for a cover relation $(a,b)$, and we denote the set of cover relations of $\Poset$ by $\Covers(\Poset)$.  An \defn{edge-labeling} of $\Poset$ is a map $\lambda\colon\Covers(\Poset)\to M$ for some set $M$.

A \defn{$k$-multichain} of $\Poset$ is a tuple $(a_{1},a_{2},\ldots,a_{k})$ with $a_{1}\leq a_{2}\leq\cdots\leq a_{k}$.  If all entries are distinct, then this tuple is a \defn{chain}.  A chain is \defn{saturated} if $a_{1}\lessdot a_{2}\lessdot\cdots\lessdot a_{k}$ and it is \defn{maximal} if it is saturated and contains a minimal and a maximal element.  A subset $A\subseteq P$ is an antichain if any two distinct members of $A$ are incomparable.

The \defn{length} of $\Poset$ is one less than the maximum cardinality of a maximal chain and is denoted by $\len(\Poset)$.  If all maximal chains have the same cardinality, then $\Poset$ is \defn{graded}.  Graded posets admit a rank function $\rk\colon P\to\mathbb{N}$ which assigns to $a\in P$ the length of the interval $[m,a]$ (regarded as a subposet of $\Poset)$ for some minimal element $m\leq a$.

The \defn{M{\"o}bius function} of $\Poset$ is recursively defined by
\begin{displaymath}
	\mu_{\Poset}(a,b) \defs \begin{cases}1, & \text{if}\;a=b,\\ -\sum\limits_{c\in P\colon a<c\leq b}\mu_{\Poset}(c,b), & \text{if}\;a<b,\\ 0, & \text{otherwise}.\end{cases}
\end{displaymath}
If $\Poset$ is bounded, then $\mu(\Poset)\defs\mu_{\Poset}(\least,\grtst)$ is the \defn{M{\"o}bius invariant} of $\Poset$.  Let $\ZetaPol_{\Poset}(q)$ denote the number of $(q-1)$-multichains of $\Poset$.  We may regard $\ZetaPol_{\Poset}$ as a polynomial, the \defn{zeta polynomial} of $\Poset$, and a classical result by G.-C.~Rota states that $\mu(\Poset)=\ZetaPol_{\Poset}(-1)$ whenever $\Poset$ is bounded~\cite{rota64foundations}.

If $\Poset_{1}=(P_{1},\leq_{1})$ and $\Poset_{2}=(P_{2},\leq_{2})$ are two posets, then their \defn{direct product} is the poset $\Poset=(P_{1}\times P_{2},\comp)$, where $(a_{1},a_{2})\comp (b_{1},b_{2})$ if and only if $a_{1}\leq_{1}b_{1}$ and $a_{2}\leq_{2}b_{2}$.  The \defn{order ideal} generated by $B\subseteq P$ is the set 
\begin{displaymath}
	\Poset_{\leq B}\defs\{a\in P\mid a\leq b\;\text{for some}\;b\in B\}.
\end{displaymath}

\subsection{Lattices}
	\label{sec:lattices}
Let $\Poset=(P,\leq)$ be a bounded poset.  The \defn{join} of $a,b\in P$ is---if it exists---the unique minimal element $a\vee b$ of the set of \defn{upper bounds} of $a$ and $b$: $\{c\in P\mid a\leq c\;\text{and}\;b\leq c\}$.  Dually, we define the \defn{meet} of $a,b\in P$ to be the unique maximal element $a\wedge b$ of the set of lower bounds of $a$ and $b$.  

If every $a,b\in P$ has a join and a meet, then $\Poset$ is a \defn{lattice}.  An \defn{atom} is an element $a\in P$ such that $(\least,a)\in\Covers(\Poset)$.  Moreover, $j\in\Poset\setminus\{\least\}$ is \defn{join irreducible} if for every $a,b\in P$ with $j=a\vee b$ it follows that $j\in\{a,b\}$.  We denote the set of join-irreducible elements of $\Poset$ by $\JI(\Poset)$.  Since $\Poset$ is by assumption finite, the join-irreducible elements of $\Poset$ are precisely those elements $j\in P$ for which there exists a unique element $a\in P$ such that $(a,j)\in\Covers(\Poset)$.  Usually, we write $j_{*}$ instead of $a$.  Dually, we may define \defn{meet-irreducible} elements, and denote the set of these elements by $\MI(\Poset)$.  A lattice is \defn{extremal} if $\bigl\lvert\JI(\Poset)\bigr\rvert=\len(\Poset)=\bigl\lvert\MI(\Poset)\bigr\rvert$; see \cite{markowsky92primes}.

We denote disjoint set union by $\uplus$.  If $\Poset=(P,\leq)$ is a lattice and $B\subseteq P$, then we consider the set
\begin{displaymath}
	P[B] \defs \Bigl(\Poset_{\leq B}\times\{0\}\Bigr)\uplus\Bigl(\bigl((P\setminus \Poset_{\leq B})\cup B\bigr)\times\{1\}\Bigr).
\end{displaymath}
The \defn{doubling} of $\Poset$ by $B$ is the poset $\Poset[B]\defs\bigl(P[B],\comp\bigr)$; see \cite{day79characterizations}.  If $B$ is an interval, then $\Poset[B]$ is a lattice~\cite{day79characterizations}, and a lattice $\Poset$ is \defn{congruence uniform} if it can be obtained from the singleton lattice by a sequence of interval doublings.  See Figure~\ref{fig:hochschild_doubling} for an illustration.

\begin{figure}
	\centering
	\begin{tikzpicture}\small
		\def\x{1.85};
		\def\y{1};
		\def\s{.5};
		\draw(1*\x,1*\y) node{\begin{tikzpicture}\small
			\draw(2,1) node[draw,fill=gray,circle,scale=\s](a1){};
			\draw(1,2) node[draw,fill=gray,circle,scale=\s](a2){};
			\draw(1,5) node{};
			\draw(a1) -- (a2);
		\end{tikzpicture}};
		\draw(1*\x,-1.5*\y) node{$\Hoch(1)$};
		\draw(1.5*\x,-.5*\y) node{$\to$};
		\draw(2.5*\x,-.5*\y) node{$\to$};
		\draw(3.5*\x,-.5*\y) node{$\to$};
		\draw(5.5*\x,-.5*\y) node{$\to$};
		\draw(2*\x,1*\y) node{\begin{tikzpicture}\small
			\draw(2,1) node[draw,circle,scale=\s](a1){};
			\draw(1,2) node[draw,fill=gray,circle,scale=\s](a2){};
			\draw(2,3) node[draw,circle,scale=\s](a3){};
			\draw(1,4) node[draw,circle,scale=\s](a4){};
			\draw(1,5) node{};
			\draw(a1) -- (a2);
			\draw(a1) -- (a3);
			\draw(a2) -- (a4);
			\draw(a3) -- (a4);
		\end{tikzpicture}};
		\draw(3*\x,1*\y) node{\begin{tikzpicture}\small
			\draw(2,1) node[draw,fill=gray,circle,scale=\s](a1){};
			\draw(1,2) node[draw,fill=gray,circle,scale=\s](a2){};
			\draw(1,3) node[draw,fill=gray,circle,scale=\s](a3){};
			\draw(2,3) node[draw,fill=gray,circle,scale=\s](a4){};
			\draw(1,4) node[draw,fill=gray,circle,scale=\s](a5){};
			\draw(1,5) node{};
			\draw(a1) -- (a2);
			\draw(a1) -- (a4);
			\draw(a2) -- (a3);
			\draw(a3) -- (a5);
			\draw(a4) -- (a5);
		\end{tikzpicture}};
		\draw(3*\x,-1.5*\y) node{$\Hoch(2)$};
		\draw(4.5*\x,1*\y) node{\begin{tikzpicture}\small
			\draw(2,1) node[draw,circle,scale=\s](a1){};
			\draw(1,2) node[draw,circle,scale=\s](a2){};
			\draw(4,2) node[draw,circle,scale=\s](a3){};
			\draw(1,3) node[draw,fill=gray,circle,scale=\s](a4){};
			\draw(2,3) node[draw,circle,scale=\s](a5){};
			\draw(3,3) node[draw,circle,scale=\s](a6){};
			\draw(1,4) node[draw,fill=gray,circle,scale=\s](a7){};
			\draw(3,4) node[draw,circle,scale=\s](a8){};
			\draw(4,4) node[draw,circle,scale=\s](a9){};
			\draw(3,5) node[draw,circle,scale=\s](a10){};
			\draw(a1) -- (a2);
			\draw(a1) -- (a3);
			\draw(a1) -- (a5);
			\draw(a2) -- (a4);
			\draw(a2) -- (a6);
			\draw(a3) -- (a6);
			\draw(a3) -- (a9);
			\draw(a4) -- (a7);
			\draw(a4) -- (a8);
			\draw(a5) -- (a7);
			\draw(a5) -- (a9);
			\draw(a6) -- (a8);
			\draw(a7) -- (a10);
			\draw(a8) -- (a10);
			\draw(a9) -- (a10);
		\end{tikzpicture}};
		\draw(6.5*\x,1*\y) node{\begin{tikzpicture}\small
			\draw(2,1) node[draw,circle,scale=\s](a1){};
			\draw(1,2) node[draw,circle,scale=\s](a2){};
			\draw(4,2) node[draw,circle,scale=\s](a3){};
			\draw(1,3) node[draw,circle,scale=\s](a4){};
			\draw(2,3) node[draw,circle,scale=\s](a5){};
			\draw(3,3) node[draw,circle,scale=\s](a6){};
			\draw(2,3.5) node[draw,circle,scale=\s](a7){};
			\draw(1,4) node[draw,circle,scale=\s](a8){};
			\draw(3,4) node[draw,circle,scale=\s](a9){};
			\draw(4,4) node[draw,circle,scale=\s](a10){};
			\draw(2,4.5) node[draw,circle,scale=\s](a11){};
			\draw(3,5) node[draw,circle,scale=\s](a12){};
			\draw(a1) -- (a2);
			\draw(a1) -- (a3);
			\draw(a1) -- (a5);
			\draw(a2) -- (a4);
			\draw(a2) -- (a6);
			\draw(a3) -- (a6);
			\draw(a3) -- (a10);
			\draw(a4) -- (a8);
			\draw(a4) -- (a7);
			\draw(a5) -- (a8);
			\draw(a5) -- (a10);
			\draw(a6) -- (a9);
			\draw(a7) -- (a9);
			\draw(a7) -- (a11);
			\draw(a8) -- (a11);
			\draw(a9) -- (a12);
			\draw(a10) -- (a12);
			\draw(a11) -- (a12);
		\end{tikzpicture}};
		\draw(6*\x,-1.5*\y) node{$\Hoch(3)$};
	\end{tikzpicture}
	\caption{Illustration of the doubling construction.  At each step, we double by the interval indicated by the highlighted elements.}
	\label{fig:hochschild_doubling}
\end{figure}
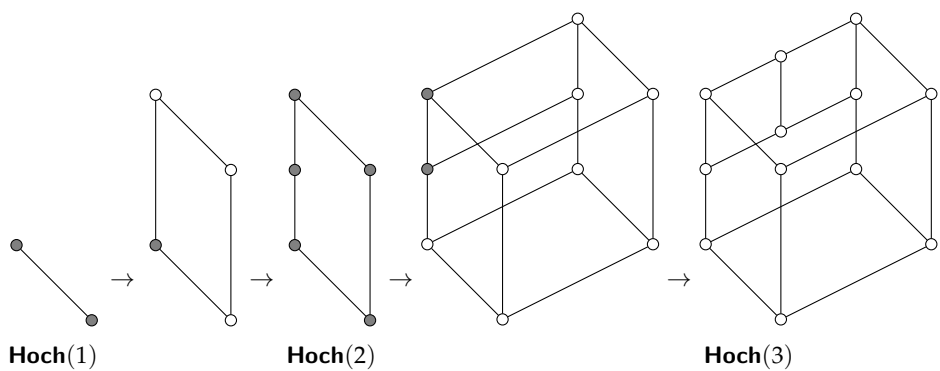

\subsection{Triwords and Hochschild lattices}
	\label{sec:triwords}
For $n>0$ we define $[n]\defs\{1,2,\ldots,n\}$.  A tuple $\ufr=(u_{1},u_{2},\ldots,u_{n})$ is a \defn{triword} of length $n$ if for all $i\in[n]$:
\begin{description}
	\item[T1\label{it:tri1}] $u_{i}\in\{0,1,2\}$;
	\item[T2\label{it:tri2}] $u_{1}\neq 2$;
	\item[T3\label{it:tri3}] if $u_{i}=0$, then $u_{j}\neq 1$ for all $j>i$.
\end{description}
Let $\Tri(n)$ denote the set of all triwords of length $n$.  A pair $(i,j)$ with $1\leq i<j\leq n$ is a \defn{$01$-pattern} in $\ufr$ if $u_{i}=0$ and $u_{j}=1$.  If $\ufr$ does not have a $01$-pattern, then it is \defn{$01$-avoiding}.  Using this notation, a triword is an element of $\{0,1,2\}^{n}$ which is $01$-avoiding and whose first letter is not a $2$.

Throughout this article, we denote elements of $\Tri(n)$ in a $\mathfrak{fraktur}$ font, and denote the $i\th$ component of such an element in a regular font with subscript $i$.  More precisely, if $\ufr\in\Tri(n)$, then $\ufr=(u_{1},u_{2},\ldots,u_{n})$.

\begin{proposition}[\cite{chapoton20some}*{Proposition~8.19}]\label{prop:triwords_size}
	For $n>0$, $\bigl\lvert\Tri(n)\bigr\rvert=2^{n-2}(n+3)$.
\end{proposition}

The numbers appearing in Proposition~\ref{prop:triwords_size} form \cite{sloane}*{A045623}.  Let $\comp$ denote the componentwise order on tuples of integers.  The partially ordered set $\Hoch(n)\defs\bigl(\Tri(n),\comp\bigr)$ is the \defn{Hochschild lattice}.  By \cite{combe20geometric}*{Section~1.2}, the poset $\Hoch(n)$ is indeed a lattice, where the join in $\Hoch(n)$ is obtained by taking componentwise maxima and the meet is obtained by taking componentwise minima and exchanging the $1$ in each resulting $01$-pattern by a $0$.  Moreover, if $(\ufr,\vfr)\in\Covers\bigl(\Hoch(n)\bigr)$, then $\ufr$ and $\vfr$ differ in exactly one component and the sum over the entries of $\vfr$ is bigger than the sum over the entries of $\ufr$.  Figure~\ref{fig:hochschild_3} shows $\Hoch(3)$.

\begin{figure}
	\centering
	\begin{tikzpicture}\small
		\def\x{2};
		\def\y{1.5};
		\def\s{.9};
		\draw(3*\x,1*\y) node[scale=\s](n1){$(0,0,0)$};
		\draw(2*\x,2*\y) node[scale=\s](n2){$(1,0,0)$};
		\draw(2*\x,3*\y) node[scale=\s](n3){$(1,1,0)$};
		\draw(3*\x,3*\y) node[scale=\s](n4){$(0,2,0)$};
		\draw(7*\x,3*\y) node[scale=\s](n5){$(0,0,2)$};
		\draw(2*\x,4*\y) node[scale=\s](n6){$(1,2,0)$};
		\draw(4*\x,4*\y) node[scale=\s](n7){$(1,1,1)$};
		\draw(6*\x,4*\y) node[scale=\s](n8){$(1,0,2)$};
		\draw(4*\x,5*\y) node[scale=\s](n9){$(1,2,1)$};
		\draw(6*\x,5*\y) node[scale=\s](n10){$(1,1,2)$};
		\draw(7*\x,5*\y) node[scale=\s](n11){$(0,2,2)$};
		\draw(6*\x,6*\y) node[scale=\s](n12){$(1,2,2)$};
		\draw(n1) -- (n2) node[fill=white,text=gray!50!red,inner sep=.5pt,scale=.75] at(2.5*\x,1.5*\y) {$\afr^{(1)}$};
		\draw(n1) -- (n4) node[fill=white,text=gray!50!red,inner sep=.5pt,scale=.75] at(3*\x,2*\y) {$\bfr^{(2)}$};
		\draw(n1) -- (n5) node[fill=white,text=gray!50!red,inner sep=.5pt,scale=.75] at(5*\x,2*\y) {$\bfr^{(3)}$};
		\draw(n2) -- (n3) node[fill=white,text=gray!50!red,inner sep=.5pt,scale=.75] at(2*\x,2.5*\y) {$\afr^{(2)}$};
		\draw(n2) -- (n8) node[fill=white,text=gray!50!red,inner sep=.5pt,scale=.75] at(4*\x,3*\y) {$\bfr^{(3)}$};
		\draw(n3) -- (n6) node[fill=white,text=gray!50!red,inner sep=.5pt,scale=.75] at(2*\x,3.5*\y) {$\bfr^{(2)}$};
		\draw(n3) -- (n7) node[fill=white,text=gray!50!red,inner sep=.5pt,scale=.75] at(3*\x,3.5*\y) {$\afr^{(3)}$};
		\draw(n4) -- (n6) node[fill=white,text=gray!50!red,inner sep=.5pt,scale=.75] at(2.5*\x,3.5*\y) {$\afr^{(1)}$};
		\draw(n4) -- (n11) node[fill=white,text=gray!50!red,inner sep=.5pt,scale=.75] at(5*\x,4*\y) {$\bfr^{(3)}$};
		\draw(n5) -- (n8) node[fill=white,text=gray!50!red,inner sep=.5pt,scale=.75] at(6.5*\x,3.5*\y) {$\afr^{(1)}$};
		\draw(n5) -- (n11) node[fill=white,text=gray!50!red,inner sep=.5pt,scale=.75] at(7*\x,4*\y) {$\bfr^{(2)}$};
		\draw(n6) -- (n9) node[fill=white,text=gray!50!red,inner sep=.5pt,scale=.75] at(3*\x,4.5*\y) {$\afr^{(3)}$};
		\draw(n7) -- (n9) node[fill=white,text=gray!50!red,inner sep=.5pt,scale=.75] at(4*\x,4.5*\y) {$\bfr^{(2)}$};
		\draw(n7) -- (n10) node[fill=white,text=gray!50!red,inner sep=.5pt,scale=.75] at(5*\x,4.5*\y) {$\bfr^{(3)}$};
		\draw(n8) -- (n10) node[fill=white,text=gray!50!red,inner sep=.5pt,scale=.75] at(6*\x,4.67*\y) {$\afr^{(2)}$};
		\draw(n9) -- (n12) node[fill=white,text=gray!50!red,inner sep=.5pt,scale=.75] at(5*\x,5.5*\y) {$\bfr^{(3)}$};
		\draw(n10) -- (n12) node[fill=white,text=gray!50!red,inner sep=.5pt,scale=.75] at(6*\x,5.5*\y) {$\bfr^{(2)}$};
		\draw(n11) -- (n12) node[fill=white,text=gray!50!red,inner sep=.5pt,scale=.75] at(6.5*\x,5.5*\y) {$\afr^{(1)}$};
	\end{tikzpicture}
	\caption{The lattice $\Hoch(3)$ labeled by \eqref{eq:jsd_labeling}.}
	\label{fig:hochschild_3}
\end{figure}
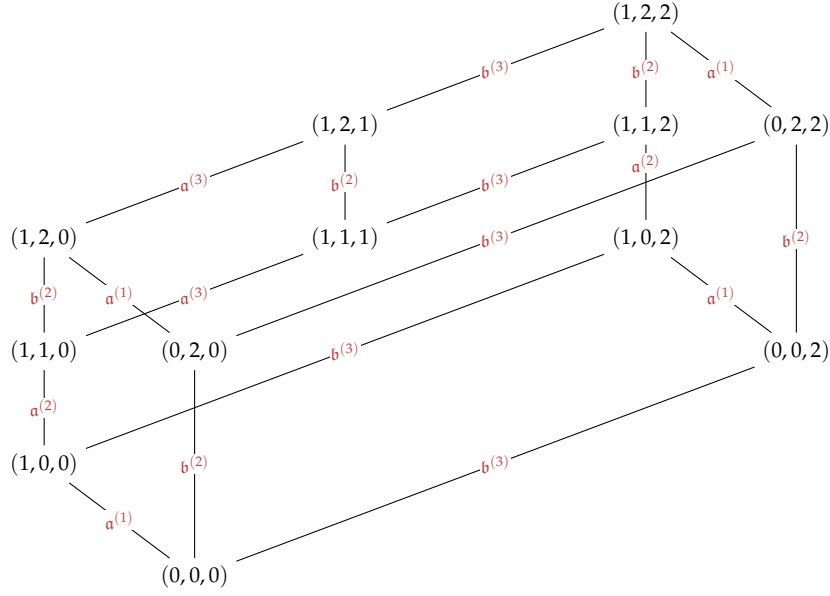

\begin{theorem}[\cite{combe20geometric}*{Theorem~2.3~and~Proposition~3.2}]\label{thm:hochschild_properties}
	For $n>0$, the lattice $\Hoch(n)$ is extremal and congruence uniform.
\end{theorem}

In fact, it was explained in \cite{combe20geometric}*{Section~2.3} that $\Hoch(n)$ can be obtained from ${\Hoch(n-1)}$ by two doublings.  One first injects $\Tri(n-1)$ into $\Tri(n)$ by appending $0$ to the end of each triword.  Then, one doubles by the full lattice, where in the doubled copy the last letter is switched from $0$ to $2$.  Finally, one doubles by the interval consisting of all triwords which have a unique zero and this zero is in the last position.  The reader is invited to label the nodes appearing in Figure~\ref{fig:hochschild_doubling} appropriately to verify this construction.

\section{The Galois graph of $\Hoch(n)$}
	\label{sec:hochschild_galois_graph}
Let $\Lattice$ be an extremal lattice, \ie 
\begin{displaymath}
	\bigl\lvert\JI(\Lattice)\bigr\rvert=\len(\Lattice)=\bigl\lvert\MI(\Lattice)\bigr\rvert.
\end{displaymath}
If $\len(\Lattice)=k$, then a maximal chain $C:\least=a_{0}\lessdot a_{1}\lessdot\cdots\lessdot a_{k}=\grtst$ of $\Lattice$ induces a linear order on both $\JI(\Lattice)$ and $\MI(\Lattice)$.  More precisely, we may label the join- and meet-irreducible elements by $j_{1},j_{2},\ldots,j_{k}$ and $m_{1},m_{2},\ldots,m_{k}$, respectively, such that for all $s\in[k]$:
\begin{equation}\label{eq:extremal_order}
	j_{1}\vee j_{2}\vee\cdots\vee j_{s}=a_{s}=m_{s+1}\wedge m_{s+2}\wedge\cdots\wedge m_{k}.
\end{equation}
The \defn{Galois graph} of $\Lattice$ is the directed graph $\Galois(\Lattice)$ with vertex set $[k]$ such that $s\to t$ if and only if $s\neq t$ and $j_{s}\not\leq m_{t}$.  If $\Lattice$ is extremal and congruence uniform, the description of $\Galois(\Lattice)$ is somewhat simpler.

\begin{lemma}[\cite{muehle18noncrossing}*{Corollary~2.15}]\label{lem:extremal_congruence_uniform_galois}
	Let $\Lattice$ be an extremal, congruence-uniform lattice of length $k$.  Let $\JI(\Lattice)$ and $\MI(\Lattice)$ be ordered as in \eqref{eq:extremal_order} with respect to some maximal chain of length $k$.  For $s,t\in[k]$, it holds that $j_{s}\not\leq m_{t}$ if and only if $s\neq t$ and $j_{t}\leq {j_{t}}_{*}\vee j_{s}$.
\end{lemma}

Before we compute the Galois graph of $\Hoch(n)$, we briefly explain that the Galois graph of an extremal lattice $\Lattice$ uniquely determines $\Lattice$; see \cites{markowsky92primes,thomas19rowmotion}.

For $k>0$, let $G=\bigl([k],E\bigr)$ be a directed graph.  A pair $(A,B)$ for $A,B\subseteq [k]$ is \defn{orthogonal} if $A\cap B=\emptyset$ and there is no $s\in A$ and no $t\in B$ such that $(s,t)\in E$.  An orthogonal pair $(A,B)$ is \defn{maximal} if $A$ and $B$ are maximal with this property.  Let $\mathsf{MO}(G)$ denote the set of maximal orthogonal pairs of $G$.

For $(A_{1},B_{1}),(A_{2},B_{2})\in\mathsf{MO}(G)$ of $G$, we set $(A_{1},B_{1})\sqsubseteq(A_{2},B_{2})$ if and only if $A_{1}\subseteq A_{2}$ (or equivalently $B_{1}\supseteq B_{2}$).  The poset $\bigl(\mathsf{MO}(G),\sqsubseteq\bigr)$ is a lattice, in which the join is computed by intersecting second components and the meet is computed by intersecting first components.

\begin{theorem}[\cite{markowsky92primes}*{Theorem~11}]\label{thm:extremal_lattice_representation}
	Every finite extremal lattice is isomorphic to the lattice of maximal orthogonal pairs of its Galois graph.  Conversely, if $G=\bigl([k],E\bigr)$ is a directed graph such that $(s,t)\in E$ only if $s>t$, then the lattice of maximal orthogonal pairs of $G$ is extremal.
\end{theorem}

We now return to studying the Galois graph of $\Hoch(n)$.  Let us consider the following triwords of length $n$:
\begin{align*}
	\afr^{(i)} & \defs (\underbrace{1,1,\ldots,1}_{i},0,0,\ldots,0) && \text{for}\;i\in[n],\\
	\bfr^{(i)} & \defs (0,0,\ldots,0,\underset{\underset{i}{\uparrow}}{2},0,0,\ldots,0) && \text{for}\;i\in\{2,3,\ldots,n\}.
\end{align*}
By construction, each of these elements is join irreducible in $\Hoch(n)$.  Inductively---using the doubling construction explained at the end of Section~\ref{sec:triwords}---we may verify that the number of join-irreducible elements in $\Hoch(n)$ is $2n-1$, which implies that every join-irreducible element of $\Hoch(n)$ is of the form $\afr^{(i)}$ or $\bfr^{(i)}$ for some appropriate choice of $i$.

Moreover, let $\ofr\defs(0,0,\ldots,0)$ and $\tfr\defs(1,2,2,\ldots,2)$ be the bottom and top elements of $\Hoch(n)$.  We may now prove Theorem~\ref{thm:hochschild_galois_graph}.

\begin{proof}[Proof of Theorem~\ref{thm:hochschild_galois_graph}]
	Recall that 
	\begin{displaymath}
		\JI\bigl(\Hoch(n)\bigr)=\bigl\{\afr^{(1)},\afr^{(2)},\ldots,\afr^{(n)},\bfr^{(2)},\bfr^{(3)},\ldots,\bfr^{(n)}\bigr\}.
	\end{displaymath}
	By construction, $\afr^{(i)}_{*}=\afr^{(i-1)}$ for $i\in\{2,3,\ldots,n\}$ and $\afr^{(1)}_{*}=\bfr^{(i)}_{*}=\ofr$ for $i\in\{2,3,\ldots,n\}$.

	By Lemma~\ref{lem:extremal_congruence_uniform_galois}, we may realize the Galois graph of $\Hoch(n)$ as a directed graph with vertex set $\JI\bigl(\Hoch(n)\bigr)$, which has an edge $\jfr\to \jfr'$ if and only if $\jfr\neq \jfr'$ and $\jfr'\leq \jfr'_{*}\vee \jfr$.
	
	If $\jfr'=\bfr^{(i)}$ for $2\leq i\leq n$, then $\bfr^{(i)}\not\leq \ofr\vee \jfr=\jfr$ for any $\jfr\in\JI\bigl(\Hoch(n)\bigr)\setminus\{\bfr^{(i)}\}$.
	
	If $\jfr'=\afr^{(1)}$, then $\afr^{(1)}\leq \ofr\vee \jfr=\jfr$ if and only if $\jfr=\afr^{(i)}$ for $i>1$.
	
	If $\jfr'=\afr^{(i)}$ for $2\leq i\leq n$, then $\afr^{(i)}\leq \afr^{(i-1)}\vee \jfr$ if and only if $\jfr=\bfr^{(i)}$ or $\jfr=\afr^{(s)}$ for $s>i$.
	
	The claim in the statement then follows by identifying $\afr^{(i)}$ with $(1,i)$ and $\bfr^{(i)}$ with $(2,i)$.
\end{proof}

Figure~\ref{fig:hochschild_galois_graphs} shows $\Galois\bigl(\Hoch(3)\bigr)$ and $\Galois\bigl(\Hoch(4)\bigr)$.  Figure~\ref{fig:hochschild_3_orthogonal_pairs} shows the lattice of orthogonal pairs of $\Galois\bigl(\Hoch(3)\bigr)$.

\begin{figure}
	\centering
	\begin{subfigure}[t]{.45\textwidth}
		\centering
		\begin{tikzpicture}\small
			\def\x{1};
			\def\y{1};
			\draw(1*\x,1*\y) node(n1){$(2,3)$};
			\draw(1*\x,2*\y) node(n2){$(1,3)$};
			\draw(1*\x,3*\y) node(n3){$(1,2)$};
			\draw(3*\x,3*\y) node(n4){$(1,1)$};
			\draw(1*\x,4*\y) node(n5){$(2,2)$};
			\draw[->](n1) -- (n2);
			\draw[->](n2) -- (n3);
			\draw[->](n2) -- (n4);
			\draw[->](n3) -- (n4);
			\draw[->](n5) -- (n3);
		\end{tikzpicture}
		\caption{$\Galois\bigl(\Hoch(3)\bigr)$.}
		\label{fig:hochschild_3_galois}
	\end{subfigure}
	\hspace*{1cm}
	\begin{subfigure}[t]{.45\textwidth}
		\centering
		\begin{tikzpicture}\small
			\def\x{1};
			\def\y{1};
			\draw(1*\x,1*\y) node(n1){$(2,3)$};
			\draw(3*\x,1*\y) node(n2){$(2,4)$};
			\draw(1*\x,2*\y) node(n3){$(1,3)$};
			\draw(3*\x,2*\y) node(n4){$(1,4)$};
			\draw(1*\x,3*\y) node(n5){$(1,2)$};
			\draw(3*\x,3*\y) node(n6){$(1,1)$};
			\draw(1*\x,4*\y) node(n7){$(2,2)$};
			\draw[->](n1) -- (n3);
			\draw[->](n2) -- (n4);
			\draw[->](n3) -- (n5);
			\draw[->](n3) -- (n6);
			\draw[->](n4) -- (n3);
			\draw[->](n4) -- (n5);
			\draw[->](n4) -- (n6);
			\draw[->](n5) -- (n6);
			\draw[->](n7) -- (n5);
		\end{tikzpicture}
		\caption{$\Galois\bigl(\Hoch(4)\bigr)$.}
		\label{fig:hochschild_4_galois}
	\end{subfigure}
	\caption{Two Galois graphs of Hochschild lattices.}
	\label{fig:hochschild_galois_graphs}
\end{figure}
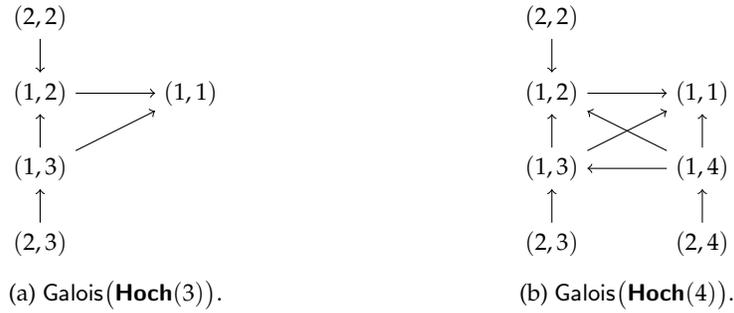

\begin{figure}
	\centering
	\begin{tikzpicture}\small
		\def\x{2};
		\def\y{1.5};
		\def\s{.85};
		\draw(3*\x,1*\y) node[scale=\s](n1){$\Bigl(-,11|12|13|22|23\Bigr)$};
		\draw(2*\x,2*\y) node[scale=\s](n2){$\Bigl(11,12|13|22|23\Bigr)$};
		\draw(2*\x,3*\y) node[scale=\s](n3){$\Bigl(11|12,13|22|23\Bigr)$};
		\draw(3*\x,3*\y) node[scale=\s](n4){$\Bigl(22,11|13|23\Bigr)$};
		\draw(7*\x,3*\y) node[scale=\s](n5){$\Bigl(23,11|12|22\Bigr)$};
		\draw(2*\x,4*\y) node[scale=\s](n6){$\Bigl(11|12|22,13|23\Bigr)$};
		\draw(4*\x,4*\y) node[scale=\s](n7){$\Bigl(11|12|13,22|23\Bigr)$};
		\draw(6*\x,4*\y) node[scale=\s](n8){$\Bigl(11|23,12|22\Bigr)$};
		\draw(4*\x,5*\y) node[scale=\s](n9){$\Bigl(11|12|13|22,23\Bigr)$};
		\draw(6*\x,5*\y) node[scale=\s](n10){$\Bigl(11|12|13|23,22\Bigr)$};
		\draw(7*\x,5*\y) node[scale=\s](n11){$\Bigl(22|23,11\Bigr)$};
		\draw(6*\x,6*\y) node[scale=\s](n12){$\Bigl(11|12|13|22|23,-\Bigr)$};
		\draw(n1) -- (n2);
		\draw(n1) -- (n4);
		\draw(n1) -- (n5);
		\draw(n2) -- (n3);
		\draw(n2) -- (n8);
		\draw(n3) -- (n6);
		\draw(n3) -- (n7);
		\draw(n4) -- (n6);
		\draw(n4) -- (n11);
		\draw(n5) -- (n8);
		\draw(n5) -- (n11);
		\draw(n6) -- (n9);
		\draw(n7) -- (n9);
		\draw(n7) -- (n10);
		\draw(n8) -- (n10);
		\draw(n9) -- (n12);
		\draw(n10) -- (n12);
		\draw(n11) -- (n12);
	\end{tikzpicture}
	\caption{The lattice of maximal orthogonal pairs of $\Galois\bigl(\Hoch(3)\bigr)$.  For brevity, we have abbreviated pairs $(s,t)$ by the word $st$, omitted set parentheses and replaced commas by vertical bars.}
	\label{fig:hochschild_3_orthogonal_pairs}
\end{figure}
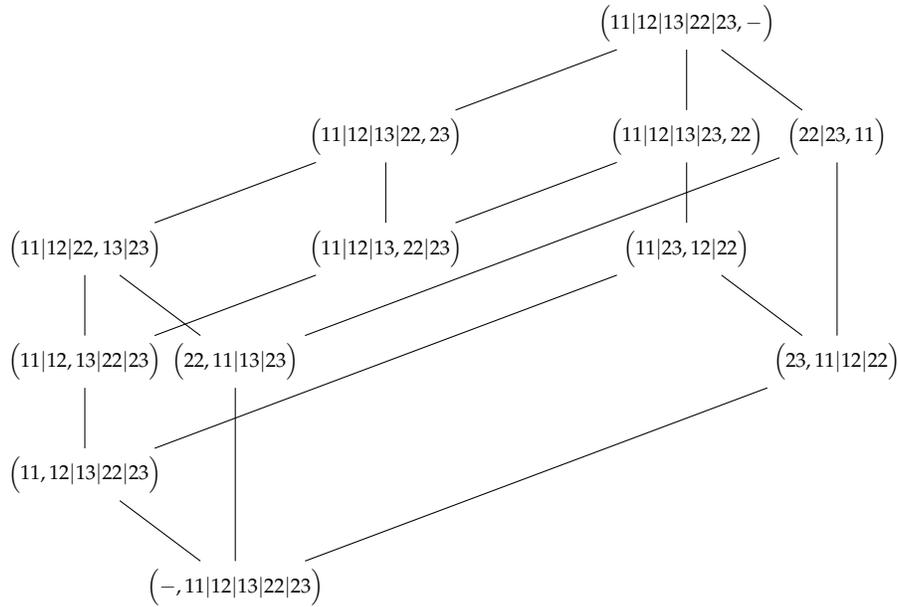

\section{The canonical join complex of $\Hoch(n)$}
	\label{sec:hochschild_join_complex}
\subsection{Join-semidistributive lattices}
	\label{sec:join_semidistributive_lattices}
A lattice $\Lattice=(L,\leq)$ is \defn{join semidistributive} if for all $a,b,c\in L$:
\begin{equation}\label{eq:join_semidistributivity}
	a\vee b=a\vee c\quad\text{implies}\quad a\vee(b\wedge c).\tag{JSD}
\end{equation}
A \defn{join representation} of $a\in L$ is a nonempty subset $A\subseteq L$ such that $\bigvee A=a$.  For two join representations $A_{1},A_{2}$ of $a$ we say that $A_{1}$ \defn{refines} $A_{2}$ if for every $a_{1}\in A_{1}$ there exists $a_{2}\in A_{2}$ such that $a_{1}\leq a_{2}$.  A join representation of $a$ is \defn{canonical} if it refines every other join representation of $a$.  A canonical join representation---when it exists---is necessarily an antichain of join-irreducible elements.

\begin{theorem}[\cite{freese95free}*{Theorem~4.2}]\label{thm:join_semidistributive_joinreps}
	A finite lattice is join semidistributive if and only if every element has a canonical join representation.
\end{theorem}

A join semidistributive lattice $\Lattice$ admits a natural edge-labeling, defined by
\begin{equation}\label{eq:jsd_labeling}
	\jsdlabeling\colon\Covers(\Lattice)\to\JI(\Lattice), \quad (a,b)\mapsto\min\{c\mid a\vee c=b\}.
\end{equation}
It is quickly verified, using \eqref{eq:join_semidistributivity}, that the codomain of $\jsdlabeling$ is indeed $\JI(\Lattice)$.  Figure~\ref{fig:hochschild_3} illustrates this labeling on $\Hoch(3)$.

This labeling has another nice characterization.  Two cover relations $(a,b),(c,d)\in\Covers(\Lattice)$ are \defn{perspective} if either $a\vee d=b$ and $a\wedge d=c$, or $c\vee b=d$ and $c\wedge b=a$.  In that case we write $(a,b)\perspective(c,d)$.

\begin{lemma}\label{lem:perspective_irreducibles}
	Let $(a,b)\in\Covers(\Lattice)$ and $j\in\JI(\Lattice)$.  If $(a,b)\perspective(j_{*},j)$, then $j\leq b$.
\end{lemma}
\begin{proof}
	If $(a,b)\perspective(j_{*},j)$, then by definition: $j\leq b$ or $b\leq j$.  If $b<j$, however, then $b\leq j_{*}$ since $j\in\JI(\Lattice)$, contradicting $b\vee j_{*}=j$.  
\end{proof}

\begin{proposition}\label{prop:jsd_labeling_perspectivity}
	Let $\Lattice$ be a join-semidistributive lattice, and let $(a,b)\in\Covers(\Lattice)$, $j\in\JI(\Lattice)$.  Then, $\jsdlabeling(a,b)=j$ if and only if $(a,b)\perspective(j_{*},j)$.
\end{proposition}
\begin{proof}
	If $\jsdlabeling(a,b)=j$, then by definition $a\vee j=b$, forcing $a\wedge j<j$.  Moreover, $j$ is minimal with the property that $a\vee j=b$, forcing $a\vee j_{*}<b$.  Since $a\lessdot b$, $a\vee j_{*}=a$, which implies $a\wedge j=j_{*}$ because $j\in\JI(\Lattice)$.  But this means $(a,b)\perspective(j_{*},j)$.
	
	Conversely, if $(a,b)\perspective(j_{*},j)$, then Lemma~\ref{lem:perspective_irreducibles} implies $j\leq b$ and thus $a\vee j=b$.  By \eqref{eq:jsd_labeling}, $\jsdlabeling(a,b)\leq j$.  However, $j_{*}\leq a$, which implies $\jsdlabeling(a,b)\not\leq j_{*}$.  Since $j\in\JI(\Lattice)$, it follows that $\jsdlabeling(a,b)=j$.
\end{proof}

We may use $\jsdlabeling$ to describe canonical join representations in $\Lattice$.

\begin{theorem}[\cite{barnard19canonical}*{Lemma~19}]\label{thm:joinreps_labels}
	Let $\Lattice=(L,\leq)$ be a join-semidistributive lattice.  The canonical join representation of $a\in L$ is $\bigl\{\jsdlabeling(b,a)\mid b\lessdot a\bigr\}$.
\end{theorem}

If $\Lattice$ satisfies both \eqref{eq:join_semidistributivity} and the dual condition (obtained by switching $\vee$ and $\wedge$), then $\Lattice$ is \defn{semidistributive}.  It is well known that if $\Lattice$ is semidistributive, then $\mu(\Lattice)\in\{-1,0,1\}$ (see for instance \cite{muehle19the}*{Theorem~2.12} for a proof).  If $\mu(\Lattice)\neq 0$, then $\Lattice$ is \defn{spherical}.  By \cite{muehle19the}*{Proposition~2.13}, $\Lattice$ is spherical if and only if $\grtst$ is the join over all atoms of $\Lattice$.  We record the observation that every congruence-uniform lattice is semidistributive.

\begin{theorem}[\cite{day79characterizations}*{Theorem~4.2}]\label{thm:congruence_uniform_is_semidistributive}
	Every congruence-uniform lattice is semidistributive.
\end{theorem}

\subsection{Canonical join representations in $\Hoch(n)$}
	\label{sec:hochschild_joinreps}
In view of Theorems~\ref{thm:hochschild_properties} and \ref{thm:congruence_uniform_is_semidistributive}, the lattice $\Hoch(n)$ is join semidistributive.  In this section, we describe the canonical join representations in $\Hoch(n)$.  Let
\begin{align*}
	f_{0}\colon\Tri(n)& \to\{1,2,\ldots,n+1\},\\
	\ufr & \mapsto \begin{cases}
		n+1, & \text{if}\;\ufr\;\text{does not contain a letter equal to}\;0,\\
		\min\{i\mid u_{i}=0\}, & \text{otherwise};
	\end{cases}\\
	l_{1}\colon\Tri(n)& \to\{0,1,2,\ldots,n\},\\
	\ufr & \mapsto \begin{cases}
		0, & \text{if}\;\ufr\;\text{does not contain a letter equal to}\;1,\\
		\max\{i\mid u_{i}=1\}, & \text{otherwise}.
	\end{cases}
\end{align*}
In other words, $f_{0}(\ufr)$ describes the position of the first zero and $l_{1}(\ufr)$ describes the position of the last $1$ in $\ufr$.  By \eqref{it:tri3}, it is always guaranteed that $l_{1}(\ufr)<f_{0}(\ufr)$.  

\begin{proposition}\label{prop:hochschild_perspective_labels}
	Let $(\ufr,\vfr)\in\Covers\bigl(\Hoch(n)\bigr)$.  Then:
	\begin{enumerate}[\rm (i)]
		\item $\jsdlabeling(\ufr,\vfr)=\afr^{(i)}$ if and only if $v_{i}=1$ and $u_{i}=0$;
		\item $\jsdlabeling(\ufr,\vfr)=\bfr^{(i)}$ if and only if $v_{i}=2$ and $u_{i}<2$.
	\end{enumerate}
\end{proposition}
\begin{proof}
	By Proposition~\ref{prop:jsd_labeling_perspectivity}, $\jsdlabeling(\ufr,\vfr)=\jfr$ if and only if $(\ufr,\vfr)\perspective(\jfr_{*},\jfr)$.  We already know that $\afr^{(i)}_{*}=\afr^{(i-1)}$ if $i>1$ and $\afr^{(1)}_{*}=\bfr^{(i)}_{*}=\ofr$.

	\medskip

	(i) Let $i\geq 1$ and suppose that $v_{i}=1$ and $u_{i}=0$.  By \eqref{it:tri3}, $v_{j}\neq 0$ for all $j<i$ and thus $u_{j}\neq 0$ for all $j<i$, since $\ufr$ and $\vfr$ differ in exactly one letter.  This implies immediately that $\afr^{(i)}\comp\vfr$ and $\afr^{(i)}\not\comp\ufr$, and $\afr^{(i-1)}\comp\ufr$.  (If $i=1$, then we set $\afr^{(i-1)}=\ofr$.)  Thus, $\afr^{(i)}\wedge\ufr=\afr^{(i-1)}$ and $\afr^{(i)}\vee\ufr=\vfr$.  By definition, $(\ufr,\vfr)\perspective(\afr^{(i-1)},\afr^{(i)})$.
	
	Conversely, suppose that $(\ufr,\vfr)\perspective(\afr^{(i-1)},\afr^{(i)})$.  By Lemma~\ref{lem:perspective_irreducibles}, $\afr^{(i)}\comp\vfr$ which implies $v_{i}=1$ and $\afr^{(i)}\not\comp\ufr$ which implies $u_{i}=0$.
	
	\medskip	
	
	(ii) Let $i>1$ and suppose that $v_{i}=2$ and $u_{i}<2$.  Then, $\bfr^{(i)}\vee\ufr=\vfr$.  Let $\mfr=\bfr^{(i)}\wedge\ufr$.  Then, $m_{j}=0$ for all $j\neq i$, which means in particular that $m_{1}=0$.  Thus $m_{i}\neq 1$ by \eqref{it:tri3}.  Since $u_{i}<2$, we must have $m_{i}=0$.  Therefore $\mfr=\ofr$, which implies $(\ufr,\vfr)\perspective(\ofr,\bfr^{(i)})$.
	
	Conversely, suppose that $(\ufr,\vfr)\perspective(\ofr,\bfr^{(i)})$.  By Lemma~\ref{lem:perspective_irreducibles}, $\bfr^{(i)}\comp\vfr$ which implies $v_{i}=2$ and $\bfr^{(i)}\not\comp\ufr$ which implies $u_{i}<2$.
\end{proof}

\begin{proposition}\label{prop:hochschild_canonical_joinrep}
	The canonical join representation of $\ufr\in\Tri(n)$ is
	\begin{equation}\label{eq:hochschild_joinreps}
		\Can(\ufr) = \Bigl\{\afr^{(i)}\mid i=l_{1}(u)\;\text{if}\;l_{1}(u)>0\Bigr\}\uplus\Bigl\{\bfr^{(i)}\mid i\in[n]\;\text{such that}\;u_{i}=2\Bigr\}.
	\end{equation}
\end{proposition}
\begin{proof}
	Let $i\in[n]$ be such that $u_{i}=2$.  If $i<f_{0}(\ufr)$, then let $\vfr^{(i)}$ be defined by decreasing the $i\th$ entry of $\ufr$ by $1$.  If $i>f_{0}(\ufr)$, then let $\wfr^{(i)}$ be defined by decreasing the $i\th$ entry of $\ufr$ by $2$.  If $l_{1}(\ufr)>0$, then let $\mfr$ be defined by decreasing the $l_{1}(\ufr)\th$ entry of $\ufr$ by $1$.  
	
	Since $\ufr\in\Tri(n)$, the tuples $\vfr^{(i)}$, $\wfr^{(i)}$, $\mfr$ are certainly triwords for the appropriate choices of $i$.  By construction, $\vfr^{(i)}\compless\ufr$, $\wfr^{(i)}<_{\mathsf{comp}}\ufr$ and $\mfr\compless\ufr$.  Note that $\wfr^{(i)}$ has a zero in the $i\th$ position and this is not the first zero in $\wfr^{(i)}$, because $i>f_{0}(\ufr)$.  By \eqref{it:tri3}, $\wfr^{(i)}\compless\ufr$, because the only potential tuple that could fit inbetween $\wfr^{(i)}$ and $\ufr$ (in componentwise order) needs to have a $1$ in position $i$.  But this would not be a triword, because it contains the $01$-pattern $\bigl(f_{0}(\ufr),i\bigr)$.
	
	Once more, by construction, every element covered by $\ufr$ is of one of these three forms.  By Proposition~\ref{prop:hochschild_perspective_labels}, $\jsdlabeling(\mfr,\ufr)=\afr^{(l_{1})}$, and $\jsdlabeling(\vfr^{(i)},\ufr)=\jsdlabeling(\wfr^{(i)},\ufr)=\bfr^{(i)}$.  Theorem~\ref{thm:joinreps_labels} finishes the proof.
\end{proof}

\begin{corollary}\label{cor:hochschild_spherical}
	For $n>0$, the lattice $\Hoch(n)$ is spherical.
\end{corollary}
\begin{proof}
	By construction, the top element of $\Hoch(n)$ is $\tfr=(1,2,\ldots,2)$.  By Proposition~\ref{prop:hochschild_canonical_joinrep}, $\Can(\tfr)=\bigl\{\afr^{(1)},\bfr^{(2)},\ldots,\bfr^{(n)}\bigr\}$, which is exactly the set of atoms of $\Hoch(n)$.  Proposition~2.13 in \cite{muehle19the} states that a semidistributive lattice is spherical if and only if the join of its atoms is the top element.  The claim thus follows.
\end{proof}

\subsection{The canonical join complex of a join-semidistributive lattice}
	\label{sec:join_semidistributive_join_complex}
Let $M$ be a finite set of \defn{vertices}.  An (abstract) simplicial complex on $M$ is a non-empty collection $\Delta(M)$ of subsets of $M$ (the \defn{faces}) such that $\{m\}\in\Delta(M)$ for all $m\in M$, and if $F\in\Delta(M)$, then $F'\in\Delta(M)$ for all $F'\subseteq F$.  The maximal faces (with respect to inclusion) are called \defn{facets}.  A simplicial complex with a unique facet is a \defn{simplex}.  A simplicial complex is \defn{pure} if all facets have the same cardinality.

Let $\Delta$ be a simplicial complex, and let $F\in\Delta$ be a face.  The \defn{link} of $F$ in $\Delta$ is the simplicial complex
\begin{displaymath}
	\link_{\Delta}(F) \defs \bigl\{G\in\Delta\mid F\cap G=\emptyset\;\text{and}\;F\cup G\in\Delta\bigr\},
\end{displaymath}
and the \defn{deletion} of $F$ in $\Delta$ is the simplicial complex
\begin{displaymath}
	\del_{\Delta}(F) \defs \bigl\{G\in\Delta\mid F\not\subseteq G\bigr\}.
\end{displaymath}
If $v_{1},v_{2},\ldots,v_{r}$ are distinct vertices of $\Delta$, then we denote by $\Delta\setminus\{v_{1},v_{2},\ldots,v_{r}\}$ the simplicial complex obtained from $\Delta$ by successively deleting the vertices $v_{1},v_{2},\ldots,v_{r}$.

Following \cites{bjorner97shellable,provan80decompositions}, a simplicial complex $\Delta$ is \defn{vertex decomposable} if either $\Delta$ is a simplex, or there exists a \defn{shedding vertex} $v\in\Delta$ satisfying the following three conditions:
\begin{description}
	\item[VD1\label{it:vd1}] $\link_{\Delta}(v)$ is vertex decomposable;
	\item[VD2\label{it:vd2}] $\del_{\Delta}(v)$ is vertex decomposable;
	\item[VD3\label{it:vd3}] no facet of $\link_{\Delta}(v)$ is a facet of $\del_{\Delta}(v)$.
\end{description}

\begin{figure}
	\centering
	\begin{subfigure}[t]{.45\textwidth}
		\centering
		\begin{tikzpicture}\small
			\def\x{1};
			\def\y{1};
			\def\s{.4};
			\coordinate (a3) at (1*\x,1*\y);
			\coordinate (b2) at (2*\x,1*\y);
			\coordinate (b3) at (3*\x,1*\y);
			\coordinate (a2) at (4*\x,1*\y);
			\coordinate (a1) at (2.5*\x,2*\y);
			\draw(a1) node[circle,fill,scale=\s]{};
			\draw(a2) node[circle,fill,scale=\s]{};
			\draw(a3) node[circle,fill,scale=\s]{};
			\draw(b2) node[circle,fill,scale=\s]{};
			\draw(b3) node[circle,fill,scale=\s]{};
			\draw(2.5*\x,2.25*\y) node[scale=.8]{$\afr^{(1)}$};
			\draw(1*\x,.75*\y) node[scale=.8]{$\afr^{(3)}$};
			\draw(2*\x,.75*\y) node[scale=.8]{$\bfr^{(2)}$};
			\draw(3*\x,.75*\y) node[scale=.8]{$\bfr^{(3)}$};
			\draw(4*\x,.75*\y) node[scale=.8]{$\afr^{(2)}$};
			\begin{pgfonlayer}{background}
				\fill[fill=white!50!gray,opacity=.7](b2) -- (a1) -- (b3) -- (b2);
				\draw(b2) -- (a1) -- (b3) -- (b2);
				\draw(a3) -- (b2);
				\draw(a2) -- (b3);
			\end{pgfonlayer}
		\end{tikzpicture}
		\caption{The complex $\CJC\bigl(\Hoch(3)\bigr)$.}
		\label{fig:hochschild_3_cjc}
	\end{subfigure}
	\hspace*{1cm}
	\begin{subfigure}[t]{.45\textwidth}
		\centering
		\begin{tikzpicture}\small
			\def\x{1};
			\def\y{1};
			\def\s{.4};
			\coordinate (a1) at (2.5*\x,2*\y);
			\coordinate (a2) at (4.1*\x,1.6*\y);
			\coordinate (a3) at (2.1*\x,1.6*\y);
			\coordinate (a4) at (1.9*\x,.4*\y);
			\coordinate (b2) at (2*\x,1*\y);
			\coordinate (b3) at (3*\x,1*\y);
			\coordinate (b4) at (3.1*\x,1.6*\y);
			\draw(a1) node[circle,fill,scale=\s]{};
			\draw(a2) node[circle,fill,scale=\s]{};
			\draw(a3) node[circle,fill,scale=\s]{};
			\draw(a4) node[circle,fill,scale=\s]{};
			\draw(b2) node[circle,fill,scale=\s]{};
			\draw(b3) node[circle,fill,scale=\s]{};
			\draw(b4) node[circle,fill,scale=\s]{};
			\draw(2.5*\x,2.25*\y) node[scale=.8]{$\afr^{(1)}$};
			\draw(4.1*\x,1.3*\y) node[scale=.8]{$\afr^{(2)}$};
			\draw(2*\x,1.85*\y) node[scale=.8]{$\afr^{(3)}$};
			\draw(2.3*\x,.4*\y) node[scale=.8]{$\afr^{(4)}$};
			\draw(1.75*\x,1*\y) node[scale=.8]{$\bfr^{(2)}$};
			\draw(3.25*\x,.8*\y) node[scale=.8]{$\bfr^{(3)}$};
			\draw(3.2*\x,1.85*\y) node[scale=.8]{$\bfr^{(4)}$};
			\begin{pgfonlayer}{background}
				\filldraw[draw=black,fill=white!50!gray,opacity=.7](a3) -- (2.3*\x,1.6*\y) -- (b2) -- (a3);
				\fill[fill=white!50!gray,opacity=.7](b2) -- (b3) -- (b4) -- (b2);
				\fill[fill=white!50!gray,opacity=.7](b2) -- (2.3*\x,1.6*\y) -- (b4) -- (b2);
				\filldraw[draw=black,fill=white!50!gray,opacity=.7](a1) -- (b2) -- (b3) -- (a1);
				\filldraw[draw=black,fill=white!50!gray,opacity=.7](a1) -- (b3) -- (b4) -- (a1);
				\filldraw[draw=black,fill=white!50!gray,opacity=.7](a2) -- (b3) -- (b4) -- (a2);
				\filldraw[draw=black,fill=white!50!gray,opacity=.7](a4) -- (b2) -- (b3) -- (a4);
				\draw[dashed](b2) -- (b4);
				\draw[dashed](2.3*\x,1.6*\y) -- (b4);
			\end{pgfonlayer}
		\end{tikzpicture}
		\caption{The complex $\CJC\bigl(\Hoch(4)\bigr)$.}
		\label{fig:hochschild_4_cjc}
	\end{subfigure}
	\caption{Two canonical join complexes of Hochschild lattices.}
	\label{fig:hochschild_join_complexes}
\end{figure}
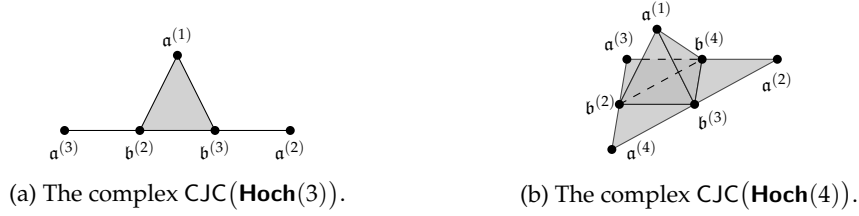

By \cite{reading15noncrossing}*{Proposition~2.2}, every subset of a canonical join representation is again a canonical join representation.  Therefore, the set of canonical join representations of a finite lattice $\Lattice$ forms a simplicial complex with vertex set $\JI(\Lattice)$; the \defn{canonical join complex} denoted by $\CJC(\Lattice)$.  If $\Lattice=(L,\leq)$ is join semidistributive, then Theorem~\ref{thm:join_semidistributive_joinreps} states that the set of faces of $\CJC(\Lattice)$ is in bijection with $L$.  Figure~\ref{fig:hochschild_join_complexes} shows $\CJC\bigl(\Hoch(3)\bigr)$ and $\CJC\bigl(\Hoch(4)\bigr)$.  Let us now prove Theorem~\ref{thm:hochschild_vertex_decomposable}.

\begin{proof}[Proof of Theorem~\ref{thm:hochschild_vertex_decomposable}]
	Throughout this proof for any face $F\in\CJC\bigl(\Hoch(n)\bigr)$, we write $\del(F)$ instead of $\del_{\CJC\bigl(\Hoch(n)\bigr)}(F)$ and $\link(F)$ instead of $\link_{\CJC\bigl(\Hoch(n)\bigr)}(F)$.

	By Proposition~\ref{prop:hochschild_canonical_joinrep}, the facets of $\CJC\bigl(\Hoch(n)\bigr)$ are
	\begin{align*}
		F_{1} & = \bigl\{\afr^{(1)}\bigr\}\uplus\bigl\{\bfr^{(j)}\mid 2\leq j\leq n\bigr\},\\
		F_{i} & = \bigl\{\afr^{(i)}\bigr\}\uplus\bigl\{\bfr^{(j)}\mid 2\leq j\leq n,j\neq i\bigr\},
	\end{align*}
	for $2\leq i\leq n$.  It follows that $\CJC\bigl(\Hoch(n)\bigr)$ is not pure.  
	
	Now, since $\afr^{(j')}\comp\afr^{(j)}$ for $j'\leq j$, each face of $\CJC\bigl(\Hoch(n)\bigr)$ contains at most one vertex of the form $\afr^{(i)}$.  For $i>1$, $\link\bigl(\afr^{(i)}\bigr)$ is thus the $(n-1)$-simplex on the vertices $\bigl\{\bfr^{(2)},\bfr^{(3)},\ldots,\bfr^{(n)}\bigr\}$.  The deletion $\del\bigl(\afr^{(i)}\bigr)$ is the subcomplex of $\CJC\bigl(\Hoch(n)\bigr)$ induced by the vertices 
	\begin{displaymath}
		\bigl\{\afr^{(1)},\ldots,\afr^{(i-1)},\afr^{(i+1)},\ldots,\afr^{(n)},\bfr^{(2)},\bfr^{(3)},\ldots,\bfr^{(n)}\bigr\}.
	\end{displaymath}
	Thus, the facets of $\del\bigl(\afr^{(i)}\bigr)$ are $F_{1}$ and $F_{j}$ for $j\neq i$.  Consequently, for $i>1$, $\afr^{(i)}$ satisfies \eqref{it:vd3}.

	It follows that the deletion $\CJC\bigl(\Hoch(n)\bigr)\setminus\bigl\{\afr^{(2)},\afr^{(3)},\ldots,\afr^{(n)}\bigr\}$ is the $n$-simplex on the vertices $\bigl\{\afr^{(1)},\bfr^{(2)},\bfr^{(3)},\ldots,\bfr^{(n)}\bigr\}$.  Since simplices are vertex decomposable, it follows that all the relevant links and deletions are vertex decomposable.  Thus, for $i>1$, $\afr^{(i)}$ satisfies \eqref{it:vd1} and \eqref{it:vd2}. 
	
	We conclude that $\CJC\bigl(\Hoch(n)\bigr)$ is vertex decomposable.
\end{proof}

\section{The core label order of $\Hoch(n)$}
	\label{sec:hochschild_core_label_order}
\subsection{The core label order of a semidistributive lattice}
	\label{sec:core_label_order}
Let $\Lattice=(L,\leq)$ be a lattice.  For $a\in L$ we define its \defn{nucleus} by
\begin{displaymath}
	a_{\downarrow} \defs a\wedge\bigwedge_{b\in L\colon b\lessdot a}{b}.
\end{displaymath}
In other words, if $a=\least$, then $a_{\downarrow}=\least$, and if $a\neq\least$, then $a_{\downarrow}$ is the meet over all elements covered by $a$.  The \defn{core} of $a$ is the interval $[a_{\downarrow},a]$ in $\Lattice$.  If $\Lattice$ is semidistributive, then we may use the labeling $\jsdlabeling$ from \eqref{eq:jsd_labeling} to define an alternate order on $L$.  The \defn{core label set} of $a$ is
\begin{displaymath}
	\Psi_{\Lattice}(a) \defs \bigl\{\jsdlabeling(b,b')\mid a_{\downarrow}\leq b\lessdot b'\leq a\bigr\}.
\end{displaymath}

\begin{proposition}\label{prop:jsd_core_labeling}
	If $\Lattice$ is semidistributive, then the assignment $a\mapsto\Psi_{\Lattice}(a)$ is injective.
\end{proposition}
\begin{proof}
	Let $a,b\in L$.  By Theorem~\ref{thm:join_semidistributive_joinreps}, $a$ and $b$ both have a canonical join representation, denoted by $\Can(a)$ and $\Can(b)$, respectively.  By Theorem~\ref{thm:joinreps_labels}, $\Can(a)\subseteq\Psi_{\Lattice}(a)$ and $\Can(b)\subseteq\Psi_{\Lattice}(b)$.
	
	Suppose that $\Psi_{\Lattice}(a)=\Psi_{\Lattice}(b)$.  If $\Psi_{\Lattice}(a)=\emptyset$, then $a=\least=b$.  Otherwise, $\Psi_{\Lattice}(a)\neq\emptyset$ implies $a\neq\least$ and $\Can(a)\neq\emptyset$.  Let $j\in\Can(a)$.  By Theorem~\ref{thm:joinreps_labels}, there exists $a'\in L$ with $a'\lessdot a$ such that $\jsdlabeling(a',a)=j$.
	
	If $j\in\Can(a)\setminus\Can(b)$, then---by assumption---$j\in\Psi_{\Lattice}(a)=\Psi_{\Lattice}(b)$, which implies that there are $b_{1},b_{2}\in L$ with $b_{\downarrow}\leq b_{1}\lessdot b_{2}<b$ and $\jsdlabeling(b_{1},b_{2})=j$.  By Proposition~\ref{prop:jsd_labeling_perspectivity}, $(a',a)\perspective(j_{*},j)$ and $(j_{*},j)\perspective(b_{1},b_{2})$.  By Lemma~\ref{lem:perspective_irreducibles}, $j\wedge a'=j_{*}=j\wedge b_{1}$, and thus, by the dual of \eqref{eq:join_semidistributivity}, $j_{*}=j\wedge(a'\vee b_{1})$.  If $a'\not\leq b_{1}$, then $a'\vee b_{1}\geq a$.  This yields the contradiction $j_{*}=j\wedge (a'\vee b_{1})=j$.  If $a'\leq b_{1}$, then $a'$ is a lower bound for $b_{2}$ and $a$ meaning that $a'$ and $j$ are comparable.  However, since $(a',a)\perspective(j_{*},j)$ we must have $j_{*}\leq a'<j$, which forces $j_{*}=a'$.  But this implies that $a=j$, and thus $\Can(a)=\{j\}=\Psi_{\Lattice}(a)=\Psi_{\Lattice}(b)=\Can(b)$.  This contradicts the choice of $j$.
	
	It follows that $\Can(a)\subseteq\Can(b)$, and symmetrically we obtain $\Can(b)\subseteq\Can(a)$.  Thus, $\Can(a)=\Can(b)$, which implies $a=b$.
\end{proof}

In view of Proposition~\ref{prop:jsd_core_labeling}, we may define a partial order $\sqsubseteq$ on $\Lattice$ by setting $a\sqsubseteq b$ if and only if $\Psi_{\Lattice}(a)\subseteq\Psi_{\Lattice}(b)$.  The poset $\CLO(\Lattice)\defs(L,\sqsubseteq)$ is the \defn{core label order} of $\Lattice$.  Figure~\ref{fig:hochschild_3_clo} shows $\CLO\bigl(\Hoch(3)\bigr)$.

\begin{figure}
	\centering
	\begin{tikzpicture}\small
		\def\x{2};
		\def\y{2};
		\def\s{.95};
		\draw(3*\x,1*\y) node[scale=\s](n1){$(0,0,0)$};
		\draw(1*\x,2*\y) node[scale=\s](n2){$(0,0,2)$};
		\draw(2*\x,2*\y) node[scale=\s](n3){$(1,1,1)$};
		\draw(3*\x,2*\y) node[scale=\s](n4){$(1,0,0)$};
		\draw(4*\x,2*\y) node[scale=\s](n5){$(0,2,0)$};
		\draw(5*\x,2*\y) node[scale=\s](n6){$(1,1,0)$};
		\draw(1*\x,3*\y) node[scale=\s](n7){$(1,0,2)$};
		\draw(2*\x,3*\y) node[scale=\s](n8){$(0,2,2)$};
		\draw(3*\x,3*\y) node[scale=\s](n9){$(1,1,2)$};
		\draw(4*\x,3*\y) node[scale=\s](n10){$(1,2,1)$};
		\draw(5*\x,3*\y) node[scale=\s](n11){$(1,2,0)$};
		\draw(3*\x,4*\y) node[scale=\s](n12){$(1,2,2)$};
		\draw(n1) -- (n2);
		\draw(n1) -- (n3);
		\draw(n1) -- (n4);
		\draw(n1) -- (n5);
		\draw(n1) -- (n6);
		\draw(n2) -- (n8);
		\draw(n2) -- (n7);
		\draw(n2) -- (n9);
		\draw(n3) -- (n9);
		\draw(n3) -- (n10);
		\draw(n4) -- (n7);
		\draw(n4) -- (n11);
		\draw(n5) -- (n8);
		\draw(n5) -- (n10);
		\draw(n5) -- (n11);
		\draw(n6) -- (n9);
		\draw(n6) -- (n11);
		\draw(n7) -- (n12);
		\draw(n8) -- (n12);
		\draw(n9) -- (n12);
		\draw(n10) -- (n12);
		\draw(n11) -- (n12);
	\end{tikzpicture}
	\caption{The lattice $\CLO\bigl(\Hoch(3)\bigr)$.}
	\label{fig:hochschild_3_clo}
\end{figure}
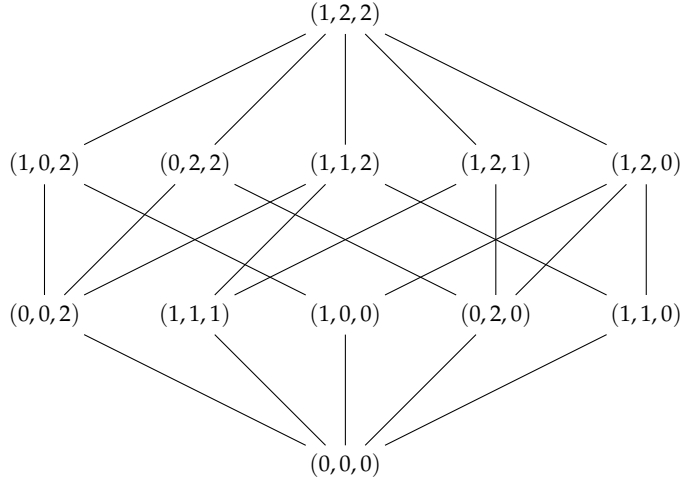

\begin{remark}
	The core label order was first considered under the name ``shard intersection order'' by N.~Reading in the context of posets of regions of hyperplane arrangements; see \cite{reading11noncrossing}.  A lattice-theoretic generalization was investigated in \cite{muehle19the}, and analogous constructions were considered for instance in \cites{bancroft11shard,clifton18canonical,garver18oriented,petersen13on}.
	
	In \cite{muehle19the}, the core label order is defined for a congruence-uniform lattice $\Lattice$ in terms of a labeling by join-irreducible elements; see \cite{muehle19the}*{Section~3.1}.  Lemma~2.6 of \cite{garver18oriented} implies that this labeling is determined by the perspectivity relation.  Hence, it agrees with our labeling $\jsdlabeling$.  The proofs of the results from \cite{muehle19the} that we use in this article depend only on this labeling and therefore extend to semidistributive lattices.  
\end{remark}

The following lemma will be useful.

\begin{lemma}\label{lem:core_labels}
	Let $\Lattice=(L,\leq)$ be a semidistributive lattice, and let $a\in L$ and $j\in\JI(\Lattice)$.  If $j\in\Psi_{\Lattice}(a)$, then $j\leq a$ and $j\not\leq a_{\downarrow}$.
\end{lemma}
\begin{proof}
	If $j\in\Psi_{\Lattice}(a)$, then there exist $b,b'\in L$ with $a_{\downarrow}\leq b\lessdot b'\leq a$ such that $(b,b')\perspective(j_{*},j)$ by Proposition~\ref{prop:jsd_labeling_perspectivity}.  Lemma~\ref{lem:perspective_irreducibles} implies that $j\leq b'\leq a$ and $j\not\leq b$.  Consequently, $j\not\leq a_{\downarrow}$.
\end{proof}


A semidistributive lattice has the \defn{intersection property} if for every $a,b\in L$ there exists $c\in L$ such that $\Psi_{\Lattice}(a)\cap\Psi_{\Lattice}(b)=\Psi_{\Lattice}(c)$.

\begin{theorem}[\cite{muehle19the}*{Theorem~1.3}]\label{thm:clo_lattice}
	The core label order of a congruence-uniform lattice $\Lattice$ is a lattice if and only if $\Lattice$ is spherical and has the intersection property.
\end{theorem}

\subsection{$\Hoch(n)$ has the intersection property}
	\label{sec:hochschild_intersection_property}
In this section, we investigate the core label order of $\Hoch(n)$, which is well defined by Theorems~\ref{thm:hochschild_properties} and \ref{thm:congruence_uniform_is_semidistributive}.  The key result is the following explicit description of the core label sets in $\Hoch(n)$.

\begin{proposition}\label{prop:hochschild_core_labels}
	The nucleus of $\ufr\in\Tri(n)$ is $\ufr_{\downarrow}=({u_{\downarrow}}_{1},{u_{\downarrow}}_{2},\ldots,{u_{\downarrow}}_{n})$ given by
	\begin{equation}\label{eq:hochschild_nucleus}
		{u_{\downarrow}}_{i} = \begin{cases}
			u_{i}-1 & \text{if either}\;i=l_{1}(\ufr),\;\text{or}\;i<l_{1}(\ufr)\;\text{and}\;u_{i}=2,\\
			0 & \text{if}\;i>l_{1}(\ufr)\;\text{and}\;u_{i}=2,\\
			u_{i} & \text{otherwise}.
		\end{cases}
	\end{equation}
	The core label set of $\ufr$ is
	\begin{equation}\label{eq:hochschild_corelabel}
		\Psi(\ufr) = \Bigl\{\afr^{(i)}\mid l_{1}(\ufr)>0\;\text{and}\;l_{1}(\ufr)\leq i<f_{0}(\ufr)\Bigr\} \uplus \Bigl\{\bfr^{(i)}\mid i\in[n], u_{i}=2\Bigr\}.
	\end{equation}\end{proposition}
\begin{proof}
	Throughout this proof we write $l_{1}$ instead of $l_{1}(\ufr)$ and $f_{0}$ instead of $f_{0}(\ufr)$.

	Let $\mfr,\vfr^{(i)},\wfr^{(i)}$ be the elements covered by $\ufr$ as constructed in the proof of Proposition~\ref{prop:hochschild_canonical_joinrep}, and let $\tilde{\ufr}$ be the componentwise minimum of all these elements.  In other words, $\tilde{\ufr}$ is obtained by subtracting $1$ from the last $1$ in $\ufr$ as well as from every $2$ in $\ufr$ occurring before the first $0$, and by subtracting $2$ from every $2$ in $\ufr$ occurring after the first $0$.  Now, if $l_{1}>0$, then for every $i\in\bigl\{l_{1}{+}1,l_{1}{+}2,\ldots,f_{0}{-}1\bigr\}$ with $u_{i}=2$, the pair $\bigl(l_{1},i\bigr)$ is a $01$-pattern in $\tilde{\ufr}$.  Thus, every $1$ occurring in such a position in $\tilde{\ufr}$ must be turned into a $0$ in order to satisfy \eqref{it:tri3}.  The resulting element is the nucleus of $\ufr$ (by definition of the meet in $\Hoch(n)$) and matches the description in \eqref{eq:hochschild_nucleus}.

	The fact that $\afr^{(l_{1})}\in\Psi(\ufr)$ and $\bfr^{(i)}\in\Psi(\ufr)$ for $i\in[n]$ such that $u_{i}=2$ follows directly from Proposition~\ref{prop:hochschild_canonical_joinrep}, since these elements constitute $\Can(\ufr)$ and $\Can(\ufr)\subseteq\Psi(\ufr)$ by Theorem~\ref{thm:joinreps_labels}.
	
	By construction, $\mfr^{(0)}=\ufr_{\downarrow}\vee\afr^{(l_{1})}$ satisfies $\ufr_{\downarrow}\compless\mfr^{(0)}\comp\ufr$.  For $i\in[f_{0}{-}l_{1}{-}1]$, we define $\mfr^{(i)}=\mfr^{(i-1)}\vee\afr^{(l_{1}+i)}$.  It follows that $\mfr^{(i-1)}\compless\mfr^{(i)}\comp\ufr$ and thus $\jsdlabeling(\mfr^{(i-1)},\mfr^{(i)})=\afr^{(l_{1}+i)}$ by Proposition~\ref{prop:hochschild_perspective_labels}.  This implies $\afr^{(l_{1}+i)}\in\Psi(\ufr)$.  See Figure~\ref{fig:core_label_sets_illustration}.

	If $i\in[n]$ is such that $u_{i}\neq 2$, then $\bfr^{(i)}\not\comp\ufr$, and thus $\bfr^{(i)}\notin\Psi(\ufr)$ by Lemma~\ref{lem:core_labels}.
	
	If $i<l_{1}$, then $u_{i}\neq 0$ by \eqref{it:tri3}, and ${u_{\downarrow}}_{i}=1$ by \eqref{eq:hochschild_nucleus}.  Consequently, $\afr^{(i)}\comp\ufr_{\downarrow}$ and thus $\afr^{(i)}\notin\Psi(\ufr)$ by Lemma~\ref{lem:core_labels}.

	If $i>f_{0}$, then $u_{i}=0$ by \eqref{it:tri3}.  Thus $\afr^{(i)}\not\comp\ufr$, and thus $\afr^{(i)}\notin\Psi(\ufr)$ by Lemma~\ref{lem:core_labels}.
\end{proof}

\begin{corollary}\label{cor:hochschild_nucleus}
	For $n>0$ and $\ufr\in\Hoch(n)$, $\ufr_{\downarrow}\in\bigl\{\ofr,\afr^{(1)},\afr^{(2)},\ldots,\afr^{(n)}\bigr\}$.
\end{corollary}
\begin{proof}
	By \eqref{eq:hochschild_nucleus}, $\ufr_{\downarrow}$ does not contain a $2$.  The claim then follows from \eqref{it:tri3}.
\end{proof}

\begin{figure}
	\centering
	\begin{tikzpicture}\small
		\def\x{2.5};
		\def\y{1};
		\def\s{.6};
		\draw(3*\x,1*\y) node[scale=\s](uu){$\ufr_{\downarrow}=(1,1,1,0,0,0,0,0,0)$};
		\draw(2.5*\x,2*\y) node[scale=\s](m0){$\mfr^{(0)}=(1,1,1,1,0,0,0,0,0)$};
		\draw(2*\x,3*\y) node[scale=\s](m1){$\mfr^{(1)}=(1,1,1,1,1,0,0,0,0)$};
		\draw(1.5*\x,4*\y) node[scale=\s](m2){$\mfr^{(2)}=(1,1,1,1,1,1,0,0,0)$};
		\draw(1*\x,7*\y) node[scale=\s](v2){$\vfr^{(2)}=(1,1,1,1,2,2,0,2,0)$};
		\draw(2*\x,7*\y) node[scale=\s](v5){$\vfr^{(5)}=(1,2,1,1,1,2,0,2,0)$};
		\draw(3*\x,7*\y) node[scale=\s](v6){$\vfr^{(6)}=(1,2,1,1,2,1,0,2,0)$};
		\draw(4*\x,7*\y) node[scale=\s](w8){$\wfr^{(8)}=(1,2,1,1,2,2,0,0,0)$};
		\draw(5*\x,7*\y) node[scale=\s](m){$\mfr=(1,2,1,0,2,2,0,2,0)$};
		\draw(3*\x,8*\y) node[scale=\s](u){$\ufr=(1,2,1,1,2,2,0,2,0)$};
		\draw(v2) -- (u) node[fill=white,text=gray!50!red,inner sep=.5pt,scale=.5] at(2*\x,7.5*\y) {$\bfr^{(2)}$};
		\draw(v5) -- (u) node[fill=white,text=gray!50!red,inner sep=.5pt,scale=.5] at(2.5*\x,7.5*\y) {$\bfr^{(5)}$};
		\draw(v6) -- (u) node[fill=white,text=gray!50!red,inner sep=.5pt,scale=.5] at(3*\x,7.5*\y) {$\bfr^{(6)}$};
		\draw(w8) -- (u) node[fill=white,text=gray!50!red,inner sep=.5pt,scale=.5] at(3.5*\x,7.5*\y) {$\bfr^{(8)}$};
		\draw(m) -- (u) node[fill=white,text=gray!50!red,inner sep=.5pt,scale=.5] at(4*\x,7.5*\y) {$\afr^{(4)}$};
		\draw[dashed](m2) -- (v2);
		\draw[dashed](m2) -- (v5);
		\draw[dashed](m2) -- (v6);
		\draw[dashed](m2) -- (w8);
		\draw(uu) -- (m0) node[fill=white,text=gray!50!red,inner sep=.5pt,scale=.5] at(2.75*\x,1.5*\y) {$\afr^{(4)}$};
		\draw(m0) -- (m1) node[fill=white,text=gray!50!red,inner sep=.5pt,scale=.5] at(2.25*\x,2.5*\y) {$\afr^{(5)}$};
		\draw(m1) -- (m2) node[fill=white,text=gray!50!red,inner sep=.5pt,scale=.5] at(1.75*\x,3.5*\y) {$\afr^{(6)}$};
		\draw[dashed](uu) -- (m);
	\end{tikzpicture}
	\caption{Illustration of the construction of the core label sets in $\Hoch(n)$.  Solid lines indicate cover relations, dashed lines indicate comparability relations.}
	\label{fig:core_label_sets_illustration}
\end{figure}
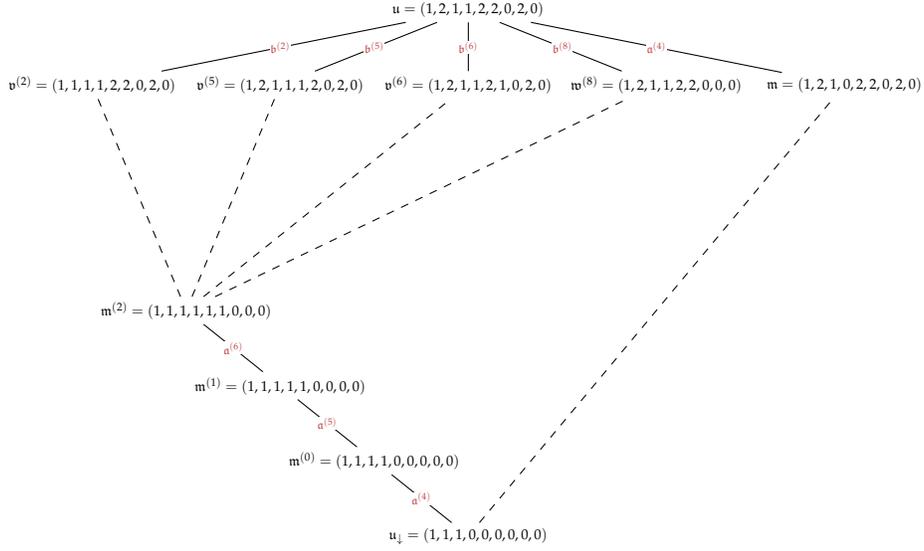

Note that we can recover $\ufr$ from $\Psi(\ufr)$.  For each $\bfr^{(i)}\in\Psi(\ufr)$, we insert a $2$ into position $i$ of an integer tuple of length $n$.  Then, the element $\afr^{(i)}\in\Psi(\ufr)$ with $i$ minimal reveals that the last $1$ must occur in position $i$.  By \eqref{it:tri3} all unfilled positions before $i$ must also contain a $1$.  All the remaining unoccupied positions must be filled with $0$s.

\begin{theorem}\label{thm:hochschild_clo_lattice}
	For $n>0$, the poset $\CLO\bigl(\Hoch(n)\bigr)$ is a lattice.
\end{theorem}
\begin{proof}
	Let $\ufr,\vfr\in\Tri(n)$, and let $U=\bigr\{i\in[n]\mid u_{i}=2\bigl\}$ and $V=\bigr\{i\in[n]\mid v_{i}=2\bigr\}$.  Consider the sets $M=U\cap V$ and
	\begin{displaymath}
		A = \Bigl\{l_{1}(\ufr),l_{1}(\ufr){+}1,\ldots,f_{0}(\ufr){-}1\Bigr\} \cap \Bigl\{l_{1}(\vfr),l_{1}(\vfr){+}1,\ldots,f_{0}(\vfr){-}1\Bigr\}.
	\end{displaymath}
	By construction, $A$ is either empty or of the form $\{i,i{+}1,\ldots,j{-}1\}$ for some $1\leq i<j\leq n$.  In particular, the set 
	\begin{displaymath}
		P = \bigl\{\afr^{(i)}\mid i\in A\bigr\}\uplus\bigl\{\bfr^{(i)}\mid i\in M\bigr\}
	\end{displaymath}
	is of the form stated in \eqref{eq:hochschild_corelabel}.  As explained after the proof of Corollary~\ref{cor:hochschild_nucleus}, we may find $\mfr\in\Tri(n)$ with $\Psi(\mfr)=P$.  
	
	This means, by definition, that $\Hoch(n)$ has the intersection property.  By Corollary~\ref{cor:hochschild_spherical}, $\Hoch(n)$ is spherical.  Then, Theorem~\ref{thm:clo_lattice} implies that $\CLO\bigl(\Hoch(n)\bigr)$ is a lattice.
\end{proof}

\subsection{Triwords and shuffles}
	\label{sec:triwords_shuffles}
We now give a combinatorial interpretation of the core label order of $\Hoch(n)$.  For integers $k,l\geq 0$, let $A=\{a_{1},a_{2},\ldots,a_{k}\}$ and $B=\{b_{1},b_{2},\ldots,b_{l}\}$ be two (disjoint) sets.  Let $\mathcal{A}=A\uplus B$ be the disjoint union of $A$ and $B$, and let $\mathcal{A}^{*}$ denote the set of words over the alphabet $\mathcal{A}$.  The empty word is denoted by $\nil$.  Let $\mathbf{u},\mathbf{v}\in\mathcal{A}^{*}$ with $\mathbf{u}=u_{1}u_{2}\cdots u_{r}$ and $\mathbf{v}=v_{1}v_{2}\cdots v_{s}$.  Then, $\mathbf{u}$ is a \defn{subword} of $\mathbf{v}$ if $r\leq s$ and there exists a sequence $1\leq i_{1}<i_{2}<\cdots<i_{r}\leq s$ such that $u_{j}=v_{i_{j}}$ for all $j\in[r]$.  

For $\mathbf{u}\in\mathcal{A}^{*}$, let $\overline{\mathbf{u}}$ denote the set of letters occurring in $\mathbf{u}$.  For $\mathbf{v}\in\mathcal{A}^{*}$, let $\mathbf{v}|\mathbf{u}$ denote the subword of $\mathbf{v}$ obtained by restricting $\mathbf{v}$ to the letters in $\overline{\mathbf{u}}$.  

Let $\wb_{A}\defs a_{1}a_{2}\cdots a_{k}$ and $\wb_{B}\defs b_{1}b_{2}\cdots b_{l}$ be two elements of $\mathcal{A}^{*}$.  We denote by $\Shuffle(A,B)$ the set of all $\wb\in\mathcal{A}^{*}$ such that $\overline{\wb}\subseteq\overline{\wb}_{A}\uplus\overline{\wb}_{B}$ and $\wb|\wb_{A}$ is a subword of $\wb_{A}$ and $\wb|\wb_{B}$ is a subword of $\wb_{B}$.  In other words, $\Shuffle(A,B)$ is the set of all \defn{shuffles} of subwords of $\wb_{A}$ and $\wb_{B}$.  

\begin{example}
	Let $A=\{a_{1},a_{2}\}$ and $B=\{b_{1}\}$.  Then,
	\begin{displaymath}
		\Shuffle(A,B) = \bigl\{\nil,a_{1},a_{2},b_{1},a_{1}a_{2},a_{1}b_{1},a_{2}b_{1},b_{1}a_{1},b_{1}a_{2},a_{1}a_{2}b_{1},a_{1}b_{1}a_{2},b_{1}a_{1}a_{2}\bigr\}.
	\end{displaymath}
	Note that $a_{2}a_{1}\notin\Shuffle(A,B)$ because it is not a subword of $\wb_{A}=a_{1}a_{2}$.
\end{example}

Clearly, $\Shuffle(A,B)$ does only depend on the size (rather than on the elements) of $A$ and $B$.  We may therefore simply write $\Shuffle(k,l)$ instead.  Following \cite{greene88posets}, we may order $\Shuffle(A,B)$ by setting $\wb_{1}\preceq\wb_{2}$ if and only if $\wb_{2}$ can be obtained from $\wb_{1}$ by removing letters of $A$ or adding letters of $B$.  The poset $\ShufflePoset(k,l)\defs\bigl(\Shuffle(k,l),\preceq\bigr)$ is in fact a lattice; the \defn{shuffle lattice}~\cite{greene88posets}*{Theorem~2.1}.

In this article, we mainly consider the shuffle lattices $\ShufflePoset(n-1,1)$, and we use the sets $A=\{2,3,\ldots,n\}$ and $B=\{\one\}$ for their construction.  Figure~\ref{fig:shuffle_21} shows $\ShufflePoset(2,1)$.

\begin{theorem}[\cite{greene88posets}*{Theorem~3.4}]\label{thm:shuffle_poset_invariants}
	Let $n>0$.  The lattice $\ShufflePoset(n-1,1)$ has $\frac{(n+1)!}{2}$ maximal chains.  Its zeta polynomial is 
	\begin{displaymath}
		\ZetaPol_{\ShufflePoset(n-1,1)}(q) = \frac{q^{n-1}}{2}\bigl((n+1)q-n+1\bigr),
	\end{displaymath}
	and its M{\"o}bius invariant is $\mu\bigl(\ShufflePoset(n-1,1)\bigr)=(-1)^{n}n$.
\end{theorem}

\begin{corollary}\label{cor:shuffles_size}
	For $n>0$, $\bigl\lvert\Shuffle(n-1,1)\bigr\rvert=2^{n-2}(n+3)$.
\end{corollary}

Recall that $A=\{a_{1},a_{2},\ldots,a_{k}\}$.  Let $\wb=w_{1}w_{2}\cdots w_{r}$ be a subword of $\wb_{A}$ and let $i\in\{0,1,\ldots,k\}$.  We define
\begin{equation}
	\wb\shuffle_{i}\one \defs \begin{cases}\wb, & \text{if}\;i=0,\\w_{1}w_{2}\cdots w_{j}\one w_{j+1}\cdots w_{r}, & \text{if}\;i>0\;\text{and}\;w_{j}=a_{i}.\end{cases}
\end{equation}
In other words, if $i>0$, then we insert $\one$ into $\wb$ after the letter $a_{i}$.  If $w_{j}\neq a_{i}$ for all $j\in[r]$, then $\wb\shuffle_{i}\one$ adds $\one$ at the beginning of $\wb$.

For $\ufr\in\Tri(n)$, let $\tau(\ufr)$ be the subword of $\wb_{A}$ consisting of the positions of $\ufr$ which do not contain the letter $2$.  We define
\begin{equation}\label{eq:triwords_shuffles}
	\sigma\colon\Tri(n)\to\Shuffle(n-1,1),\quad \ufr\mapsto\tau(\ufr)\shuffle_{l_{1}(\ufr)}\one.
\end{equation}
Table~\ref{tab:sigma_bijection_21} illustrates this map in the case $n=3$.

\begin{proposition}\label{prop:bijection_triwords_shuffles}
	The map $\sigma$ from \eqref{eq:triwords_shuffles} is a bijection.
\end{proposition}
\begin{proof}
	By \eqref{it:tri3}, any $\ufr\in\Tri(n)$ is uniquely determined by its positions of the $2$s and the position of its last $1$.  Thus, $\sigma$ is injective.  
	
	Conversely, let $\ab\in\Shuffle(n-1,1)$.  Let $M$ denote the set of letters of $\ab$ which are different from $\one$.  Then, we set $u_{i}=2$ for all $i\in\{2,3,\ldots,n\}\setminus M$.  If $\one$ is not a letter of $\ab$, then we set $u_{i}=0$ if $i=1$ or $i\in M$.  Otherwise, suppose that $\one$ is the $(j+1)\st$ letter of $\ab$ and let $a_{j}$ be the letter of $\ab$ directly preceding $\one$.  We set $u_{i}=1$ if $i=1$ or if $i\in M$ such that $i\leq a_{j}$.  We set $u_{i}=0$ if $i\in M$ with $i>a_{j}$.  By construction, $\ufr_{\ab}=(u_{1},u_{2},\ldots,u_{n})$ satisfies \eqref{it:tri1}--\eqref{it:tri3}.  Moreover, it is immediately clear that $\sigma(\ufr_{\ab})=\ab$.
\end{proof}

\begin{example}
	Let $\ab=3\;4\;\one\;6\;9\;10\in\Shuffle(9,1)$.  We have $M=\{3,4,6,9,10\}$, and we see that $\ab$ contains the letter $\one$ in position $j+1=3$, and we obtain $a_{2}=4$.  Then, $\{2,3,\ldots,10\}\setminus M=\{2,5,7,8\}$, and we find $\ufr_{\ab}=(1,2,1,1,2,0,2,2,0,0)\in\Tri(10)$.
	
	We have $l_{1}(\ufr_{\ab})=4$ and $\tau(\ufr_{\ab})=3\;4\;6\;9\;10$.  Thus, $\sigma(\ufr_{\ab})=\tau(\ufr_{\ab})\shuffle_{4}\one = 3\;4\;\one\;6\;9\;10=\ab$.
\end{example}

\begin{table}
	\centering
	\begin{tabular}{c|c|c|c}
		$\ufr\in\Tri(3)$ & $\tau(\ufr)$ & $l_{1}(\ufr)$ & $\sigma(\ufr)=\tau(\ufr)\shuffle_{l_{1}(\ufr)}\one$\\
		\hline\hline
		$(0,0,0)$ & $23$ & $0$ & $23$ \\
		$(0,0,2)$ & $2$ & $0$ & $2$ \\ 
		$(0,2,0)$ & $3$ & $0$ & $3$ \\
		$(0,2,2)$ & $\nil$ & $0$ & $\nil$ \\
		$(1,0,0)$ & $23$ & $1$ & $\one 23$ \\
		$(1,0,2)$ & $2$ & $1$ & $\one 2$ \\
		$(1,1,0)$ & $23$ & $2$ & $2\one 3$ \\
		$(1,1,1)$ & $23$ & $3$ & $23\one$ \\
		$(1,1,2)$ & $2$ & $2$ & $2\one$ \\
		$(1,2,0)$ & $3$ & $1$ & $\one 3$ \\
		$(1,2,1)$ & $3$ & $3$ & $3\one$ \\
		$(1,2,2)$ & $\nil$ & $1$ & $\one$ \\
	\end{tabular}
	\caption{The bijection $\sigma$ illustrated for $n=3$.}
	\label{tab:sigma_bijection_21}
\end{table}

\begin{figure}
	\centering
	\begin{tikzpicture}\small
		\def\x{2};
		\def\y{2};
		\def\s{.95};
		\draw(3*\x,1*\y) node[scale=\s](n1){$23$};
		\draw(1*\x,2*\y) node[scale=\s](n2){$2$};
		\draw(2*\x,2*\y) node[scale=\s](n3){$23\one$};
		\draw(3*\x,2*\y) node[scale=\s](n4){$\one23$};
		\draw(4*\x,2*\y) node[scale=\s](n5){$3$};
		\draw(5*\x,2*\y) node[scale=\s](n6){$2\one3$};
		\draw(1*\x,3*\y) node[scale=\s](n7){$\one2$};
		\draw(2*\x,3*\y) node[scale=\s](n8){$\nil$};
		\draw(3*\x,3*\y) node[scale=\s](n9){$2\one$};
		\draw(4*\x,3*\y) node[scale=\s](n10){$3\one$};
		\draw(5*\x,3*\y) node[scale=\s](n11){$\one3$};
		\draw(3*\x,4*\y) node[scale=\s](n12){$\one$};
		\draw(n1) -- (n2);
		\draw(n1) -- (n3);
		\draw(n1) -- (n4);
		\draw(n1) -- (n5);
		\draw(n1) -- (n6);
		\draw(n2) -- (n8);
		\draw(n2) -- (n7);
		\draw(n2) -- (n9);
		\draw(n3) -- (n9);
		\draw(n3) -- (n10);
		\draw(n4) -- (n7);
		\draw(n4) -- (n11);
		\draw(n5) -- (n8);
		\draw(n5) -- (n10);
		\draw(n5) -- (n11);
		\draw(n6) -- (n9);
		\draw(n6) -- (n11);
		\draw(n7) -- (n12);
		\draw(n8) -- (n12);
		\draw(n9) -- (n12);
		\draw(n10) -- (n12);
		\draw(n11) -- (n12);
	\end{tikzpicture}
	\caption{The lattice $\ShufflePoset(2,1)$.}
	\label{fig:shuffle_21}
\end{figure}

We now prove the main result of this section, which states that the core label order of $\Hoch(n)$ is isomorphic to $\ShufflePoset(n-1,1)$.

\begin{proof}[Proof of Theorem~\ref{thm:hochschild_clo_shuffleposet}]
	We prove that the map $\sigma$ from \eqref{eq:triwords_shuffles} is an isomorphism from $\CLO\bigl(\Hoch(n)\bigr)$ to $\ShufflePoset(n-1,1)$.  Let $\ufr,\ufr'\in\Tri(n)$ and let $\ab=\sigma(\ufr)$ and $\ab'=\sigma(\ufr')$.  Throughout this proof, we write $\Psi(\ufr)$ instead of $\Psi_{\Hoch(n)}(\ufr)$.
	
	\medskip
	
	First, suppose that $\Psi(\ufr)\subseteq\Psi(\ufr')$.  By Proposition~\ref{prop:hochschild_core_labels}, the positions of the $2$s in $\ufr$ form a subset of the positions of the $2$s in $\ufr'$.  We distinguish two cases.
	
	(i) If $l_{1}(\ufr)=0$, then $\ufr$ does not contain a $1$ and $\ab$ does not contain $\one$.  
	
	If $l_{1}(\ufr')=0$, then $\ab'$ does not contain $\one$, and $\Psi(\ufr)\subseteq\Psi(\ufr')$ implies that $\ab'$ is obtained from $\ab$ by (potentially) removing elements of $\{2,3,\ldots,n\}$, which implies $\ab\preceq\ab'$.  
	
	If $l_{1}(\ufr')>0$, then $\ab'$ contains $\one$.  As before, $\Psi(\ufr)\subseteq\Psi(\ufr')$ implies that $\ab'$ is obtained from $\ab$ by adding $\one$ and (potentially) removing elements of $\{2,3,\ldots,n\}$, which implies $\ab\preceq\ab'$.
	
	(ii) If $l_{1}(\ufr)>0$, then $\Psi(\ufr)\subseteq\Psi(\ufr')$ implies that $l_{1}(\ufr')>0$ and $l_{1}(\ufr)\geq l_{1}(\ufr')$ and $f_{0}(\ufr)\leq f_{0}(\ufr')$.  By construction, for every $i\in[n]$ with $l_{1}(\ufr')\leq i<l_{1}(\ufr)$ or $f_{0}(\ufr)\leq i<f_{0}(\ufr')$ we must have $u'_{i}=2$.  Thus, $\ab'$ is obtained from $\ab$ by removing elements of $\{2,3,\ldots,n\}$, which implies $\ab\preceq\ab'$.
	
	\medskip
	
	Conversely, suppose that $(\ab,\ab')\in\Covers\bigl(\ShufflePoset(n-1,1)\bigr)$.  There are two cases.
	
	(i) There exists $j\in\{2,3,\ldots,n\}$ which is contained in $\ab$ but not in $\ab'$.  By Proposition~\ref{prop:bijection_triwords_shuffles}, $u_{j}\neq 2$ and $u'_{j}=2$.  Then, $l_{1}(\ufr)\geq l_{1}(\ufr')$, $f_{0}(\ufr)\leq f_{0}(\ufr')$ and $\{i\in[n]\mid u_{i}=2\}\subsetneq\{i\in[n]\mid u'_{i}=2\}$.  By Proposition~\ref{prop:hochschild_core_labels}, $\Psi(\ufr)\subsetneq\Psi(\ufr')$.
	
	(ii) $\one$ is not contained in $\ab$ and $\one$ is contained in $\ab'$, say in position $j+1$.  By Proposition~\ref{prop:bijection_triwords_shuffles}, $\ufr$ does not contain $1$ and thus $l_{1}(\ufr)=0$.  Moreover, $\bigr\{i\in[n]\mid u_{i}=2\bigr\}=\bigl\{i\in[n]\mid u'_{i}=2\bigr\}$.  Since $\one$ is contained in $\ab'$, $\ufr'$ contains $1$ and thus $l_{1}(\ufr')>0$.  By Proposition~\ref{prop:hochschild_core_labels}, $\Psi(\ufr)\subsetneq\Psi(\ufr')$.
	
	We conclude that $\ab\preceq\ab'$ implies $\Psi(\ufr)\subseteq\Psi(\ufr')$, which finishes the proof.
\end{proof}

\begin{proposition}\label{prop:hochschild_clo_ranks}
	For $n>0$, the lattice $\CLO\bigl(\Hoch(n)\bigr)$ is graded and for $0\leq k\leq n$ its number of elements of rank $k$ is
	\begin{displaymath}
		\binom{n}{k}+(n-k)\binom{n-1}{k-1}.
	\end{displaymath}
\end{proposition}
\begin{proof}
	It was shown in \cite{greene88posets}*{Section~2} that $\ShufflePoset(n-1,1)$ is graded.  In order to describe the rank function, let $\ab\in\Shuffle(n-1,1)$ and write $w(\ab)$ for the number of elements of $\{2,3,\ldots,n\}$ contained in $\ab$.  Then, the rank of $\ab$ in $\ShufflePoset(n-1,1)$ is
	\begin{displaymath}
		\rk(\ab) = n-1 - w(\ab) + \begin{cases}1, & \text{if}\;\ab\;\text{contains}\;\one,\\0, & \text{otherwise}.\end{cases}
	\end{displaymath}
	In view of the isomorphism $\sigma$ from \eqref{eq:triwords_shuffles}, this translates to $\CLO\bigl(\Hoch(n)\bigr)$ as follows:
	\begin{equation}\label{eq:hochschild_clo_rank}
		\rk(\ufr) = \bigl\lvert\{i\mid u_i=2\}\bigr\rvert + \begin{cases}1 & \text{if}\;l_{1}(\ufr)>0,\\0 & \text{otherwise}.\end{cases}
	\end{equation}

	Now, let $R(n,k)$ denote the number of elements of rank $k$ in $\CLO\bigl(\Hoch(n)\bigr)$.  If $n=1$, then $R(1,0)=1=R(1,1)$. If $n>1$, then $R(n,0)=1$.  Now let $k\geq 1$ and let $\ufr\in\Tri(n)$ with $\rk(\ufr)=k$.  If $l_{1}(\ufr)=0$, then $\ufr$ must contain exactly $k$ letters equal to $2$ in the last $n-1$ positions, because by \eqref{it:tri2} $\ufr$ cannot start with a $2$.  If $l_{1}(\ufr)=1$, then $\ufr$ must have exactly $k-1$ letters $2$ in the last $n-1$ positions.  If $l_{1}(\ufr)>1$, then $\ufr$ must have at least two letters equal to $1$ and $k-1$ letters equal to $2$.  We obtain
	\begin{align*}
		R(n,k) & = \binom{n-1}{k} + \binom{n-1}{k-1} + (n-1)\binom{n-2}{k-1}\\
		& = \binom{n-1}{k} + (n-k+1)\binom{n-1}{k-1}\\
		& = \binom{n}{k} + (n-k)\binom{n-1}{k-1}.\qedhere
	\end{align*}
\end{proof}

\begin{corollary}\label{cor:hochschild_covers_clo_rank}
	Let $\ufr\in\Tri(n)$.  The rank of $\ufr$ in $\CLO\bigl(\Hoch(n)\bigr)$ equals the number of elements covered by $\ufr$ in $\Hoch(n)$.
\end{corollary}
\begin{proof}
	The number of elements covered by $\ufr$ in $\Hoch(n)$ is $\bigl\lvert\Can(\ufr)\bigr\rvert$ by Proposition~\ref{prop:hochschild_canonical_joinrep}.  By \eqref{eq:hochschild_joinreps}, this equals the number of $2$s in $\ufr$ plus one if and only if $l_{1}(\ufr)>0$.  In view of \eqref{eq:hochschild_clo_rank}, this number is precisely the rank of $\ufr$ in $\CLO\bigl(\Hoch(n)\bigr)$.
\end{proof}

\section{$M$-, $H$- and $F$-triangles for $\Hoch(n)$}
	\label{sec:chapoton_triangles}
\subsection{Two rank-generating polynomials}
	\label{sec:rank_generating}
We have just seen that $\CLO\bigl(\Hoch(n)\bigr)$ is ranked by the rank function $\rk$.  We abbreviate the M{\"o}bius function of $\CLO\bigl(\Hoch(n)\bigr)$ by $\mu_{n}$, and consider the following two polynomials.  The \defn{rank-generating polynomial}, defined by
\begin{displaymath}
	r_{n}(x) \defs \sum_{\ufr\in\Tri(n)}x^{\rk(\ufr)},
\end{displaymath}
and the (reverse) \defn{characteristic polynomial} of $\CLO\bigl(\Hoch(n)\bigr)$, essentially a weighted version of $r_{n}(x)$:
\begin{displaymath}
	\tilde{\chi}_{n}(x) \defs \sum_{\ufr\in\Tri(n)}\mu_{n}(\ofr,\ufr)x^{\rk(\ufr)}.
\end{displaymath}

\begin{remark}\label{rem:rank_polynomials}
	Clearly, we may define rank-generating and reverse characteristic polynomials verbatim for any graded, bounded poset.
\end{remark}

Using the rank numbers of $\CLO\bigl(\Hoch(n)\bigr)$ from Proposition~\ref{prop:hochschild_clo_ranks}, we may compute a closed formula for $r_{n}(x)$ and \cite{greene88posets}*{Theorem~3.4} provides a closed formula for $\tilde{\chi}_{n}(x)$.

\begin{proposition}\label{prop:hochschild_clo_rank_polynomials}
	For $n>0$, 
	\begin{align*}
		r_{n}(x) & = (x+1)^{n-2}\bigl(x^{2}+(n+1)x+1\bigr),\\
		\tilde{\chi}_{n}(x) & = (1-x)^{n-1}(1-nx).
	\end{align*}
\end{proposition}
\begin{proof}
	By Proposition~\ref{prop:hochschild_clo_ranks}, we obtain
	\begin{align*}
		r_{n}(x) & = \sum_{k=0}^{n}\left(\binom{n}{k} + (n-k)\binom{n-1}{k-1}\right)x^{k}\\
		& = (x+1)^{n} + \sum_{k=1}^{n-1}(n-k)\binom{n-1}{k-1}x^{k}\\
		& = (x+1)^{n} + \sum_{k=0}^{n-2}(k+1)\binom{n-1}{k+1}x^{k+1}\\
		& = (x+1)^{n} + \sum_{k=0}^{n-2}(n-1)\binom{n-2}{k}x^{k+1}\\
		& = (x+1)^{n} + (x+1)^{n-2}(n-1)x,
	\end{align*}
	which equals the desired formula.
	
	Theorem~1.3 in \cite{greene88posets} implies that the reverse characteristic polynomial of $\ShufflePoset(a,b)$ equals
	\begin{equation}\label{eq:shuffle_charpol}
		\tilde{\chi}_{\ShufflePoset(a,b)}(x) = (-1)^{a+b}\sum_{j\geq 0}\binom{a}{j}\binom{b}{j}(x-1)^{a+b-j}x^j.
	\end{equation}
	By Theorem~\ref{thm:hochschild_clo_shuffleposet}, $\CLO\bigl(\Hoch(n)\bigr)\cong\ShufflePoset(n-1,1)$, which yields the claim.
\end{proof}

\subsection{The $M$-triangle of $\CLO\bigl(\Hoch(n)\bigr)$}
	\label{sec:hochschild_m_triangle}
It is straightforward to define a refinement of the (reverse) characteristic polynomial; the \defn{$M$-triangle} of $\CLO\bigl(\Hoch(n)\bigr)$:
\begin{displaymath}
	M_{n}(x,y) \defs \sum_{\ufr,\vfr\in\Tri(n)}\mu_{n}(\ufr,\vfr)x^{\rk(\ufr)}y^{\rk(\vfr)}.
\end{displaymath}

For $\ufr\in\Tri(n)$, let $\tilde{\chi}_{[\ufr,\tfr]}(x)$ denote the (reverse) characteristic polynomial of the interval $[\ufr,\tfr]$ in $\CLO\bigl(\Hoch(n)\bigr)$.  
The following relations are immediate.

\begin{lemma}\label{lem:mtriangle_decomposition}
	For $n>0$,
	\begin{align*}
		M_{n}(x,y) & = \sum_{\ufr\in\Tri(n)}(xy)^{\rk(\ufr)}\tilde{\chi}_{[\ufr,\tfr]}(y),\\
		\tilde{\chi}_{n}(x) & = M_{n}(0,x).
	\end{align*}
\end{lemma}
\begin{proof}
	For the first equality, we have:
	\begin{align*}
		M_{n}(x,y) & = \sum_{\ufr,\vfr\in\Tri(n)}\mu_{n}(\ufr,\vfr)x^{\rk(\ufr)}y^{\rk(\vfr)}\\
		 & = \sum_{\ufr\in\Tri(n)}(xy)^{\rk(\ufr)}\sum_{\vfr\in\Tri(n)\colon\Psi(\ufr)\subseteq\Psi(\vfr)}\mu_{n}(\ufr,\vfr)y^{\rk(\vfr)-\rk(\ufr)}\\
		 & = \sum_{\ufr\in\Tri(n)}(xy)^{\rk(\ufr)}\tilde{\chi}_{[\ufr,\tfr]}(y).
	\end{align*}
	The second equality follows if we evaluate $0^{0}=1$.
\end{proof}

\begin{theorem}\label{thm:hochschild_clo_mtriangle}
	For $n>0$,
	\begin{displaymath}
		M_{n}(x,y) = (xy-y+1)^{n-2}\Bigl((n+1)\bigl((x-1)y-xy^{2}\bigr)+(n+x^{2})y^{2}+1\Bigr).
	\end{displaymath}
\end{theorem}
\begin{proof}
	Let $\ufr\in\Tri(n)$ with $\rk(\ufr)=k$, and let $[\ufr,\tfr]$ denote the interval between $\ufr$ and $\tfr$ in $\CLO\bigl(\Hoch(n)\bigr)$ regarded as an induced subposet.  In view of the isomorphism from Theorem~\ref{thm:hochschild_clo_shuffleposet}, it is straightforward to verify that 
	\begin{displaymath}
		[\ufr,\tfr] \cong \begin{cases}\CLO\bigl(\Hoch(n-k)\bigr), & \text{if}\;l_{1}(\ufr)=0,\\\Bool(n-k), & \text{otherwise}.\end{cases}
	\end{displaymath}
	
	As shown in the proof of Proposition~\ref{prop:hochschild_clo_ranks}, there are $\binom{n-1}{k}$ triwords of rank $k$ (in $\CLO\bigl(\Hoch(n)\bigr)$) which do not contain the letter $1$, and $(n-k+1)\binom{n-1}{k-1}$ triwords of rank $k$ which do.  
	
	Since $\Bool(n)\cong\ShufflePoset(n,0)$ by construction, Equation~\eqref{eq:shuffle_charpol} implies that
	\begin{displaymath}
		\tilde{\chi}_{\Bool(n)}(x) = (1-x)^{n}.
	\end{displaymath}
	Now, using Lemma~\ref{lem:mtriangle_decomposition}, we obtain
	\begin{align*}
		M_{n}(x,y) & = \sum_{\ufr\in\Tri(n)}(xy)^{\rk(\ufr)}\tilde{\chi}_{[\ufr,\tfr]_{\CLO}}(y)\\
		& = \sum_{k=0}^{n}(xy)^{k}\left(\binom{n-1}{k}\tilde{\chi}_{\CLO\bigl(\Hoch(n-k)\bigr)}(y)\right.\\
		& \kern2cm \left. + (n-k+1)\binom{n-1}{k-1}\tilde{\chi}_{\Bool(n-k)}(y)\right)\\
		& = \sum_{k=0}^{n}(xy)^{k}\left(\binom{n-1}{k}(1-y)^{n-k-1}\bigl(1-(n-k)y\bigr)\right.\\
		& \kern2cm \left. + (n-k+1)\binom{n-1}{k-1}(1-y)^{n-k}\right)\\
		& = \sum_{k=0}^{n-1}\binom{n-1}{k}(xy)^{k}(1-y)^{n-k-1}\bigl(1-(n-k)y\bigr)\\
		& \kern2cm + xy\sum_{k=0}^{n-1}\binom{n-1}{k}(xy)^{k}(1-y)^{n-k-1}(n-k).
	\end{align*}
	Let us treat the two sums separately.  We define
	\begin{align*}
		S_{1}(x,y) & \defs \sum_{k=0}^{n-1}\binom{n-1}{k}(xy)^{k}(1-y)^{n-k-1}\bigl(1-(n-k)y\bigr),\\
		S_{2}(x,y) & \defs \sum_{k=0}^{n-1}\binom{n-1}{k}(xy)^{k}(1-y)^{n-k-1}(n-k),
	\end{align*}
	so that
	\begin{equation}\label{eq:target}
		M_{n}(x,y) = S_{1}(x,y) + xyS_{2}(x,y).
	\end{equation}
	We see right away that 
	\begin{equation}\label{eq:simpler}
		S_{1}(x,y) = (xy-y+1)^{n-1} - yS_{2}(x,y).
	\end{equation}
	Partial differentiation and the Binomial Theorem yield
	\begin{align*}
		\sum_{k=0}^{n-1}\binom{n-1}{k}k(xy)^{k}(1-y)^{n-k-1} & = x\left(\sum_{k=0}^{n-1}\binom{n-1}{k}kx^{k-1}y^{k}(1-y)^{n-k-1}\right)\\
		& = x\left(\frac{d}{dx}\sum_{k=0}^{n-1}\binom{n-1}{k}(xy)^{k}(1-y)^{n-k-1}\right)\\
		& = x\left(\frac{d}{dx}(xy-y+1)^{n-1}\right)\\
		& = (n-1)xy(xy-y+1)^{n-2},
	\end{align*}
	which implies
	\begin{align*}
		S_{2}(x,y) & = n\sum_{k=0}^{n-1}\binom{n-1}{k}(xy)^{k}(1-y)^{n-k-1} - \sum_{k=0}^{n-1}\binom{n-1}{k}k(xy)^{k}(1-y)^{n-k-1}\\
		& = n(xy-y+1)^{n-1} - (n-1)xy(xy-y+1)^{n-2}\\
		& = (xy-y+1)^{n-2}(xy-ny+n).
	\end{align*}
	Combining this with \eqref{eq:target} and \eqref{eq:simpler} gives
	\begin{align*}
		M_{n} & (x,y) = S_{1}(x,y) + xyS_{2}(x,y)\\
		& = (xy-y+1)^{n-1} + (xy-y)S_{2}(x,y)\\
		& = (xy-y+1)^{n-1} + (xy-y)(xy-y+1)^{n-2}(xy-ny+n)\\
		& = (xy-y+1)^{n-2}\Bigl((n+1)\bigl((x-1)y-xy^{2}\bigr) + (n+x^{2})y^{2} + 1\Bigr).\qedhere
	\end{align*}
\end{proof}

\begin{example}\label{ex:hochschild_3_m_triangle}
	Figure~\ref{fig:hochschild_3_clo} shows the lattice $\CLO\bigl(\Hoch(3)\bigr)$.  It has five elements of rank $1$, two elements of rank $2$ inducing an ideal with five elements and three elements of rank $2$ inducing an ideal with four elements.  Thus, the M{\"o}bius invariant of $\CLO\bigl(\Hoch(3)\bigr)$ is $3$.  Finally, since there are twelve cover relations connecting elements of rank $1$ and rank $2$, we obtain
	\begin{align*}
		M_{3}(x,y) & = 1 + 5xy + 5x^{2}y^{2} + x^{3}y^{3} - 5y + 7y^{2} - 3y^{3} - 12xy^{2} + 7xy^{3} - 5x^{2}y^{3}\\
		& = (xy-y+1)\bigl(4((x-1)y-xy^{2})+(3+x^{2})y^{2}+1\bigr).
	\end{align*}
\end{example}

\begin{remark}\label{rem:boolean_mtriangle}
	Of course, we may as well define the $M$-triangle for any graded poset.  Indeed, since every interval in $\Bool(n)$ is isomorphic to a smaller Boolean lattice, Lemma~\ref{lem:mtriangle_decomposition} yields
	\begin{align*}
		M_{\Bool(n)}(x,y) & = \sum_{A\subseteq[n]}(xy)^{\lvert A\rvert}\tilde{\chi}_{\Bool(n-\lvert A\rvert)}(y)\\
		& = \sum_{k=0}^{n}\binom{n}{k}(xy)^{k}(1-y)^{n-k}\\
		& = (xy-y+1)^{n}.
	\end{align*}
\end{remark}

\subsection{$F$ and $H$-triangles for $\Hoch(n)$}
	\label{sec:hochschild_fh_triangle}
One of the first occurrences of the $M$-triangle of a graded poset is perhaps \cite{chapoton04enumerative}, where such a polynomial was introduced for the lattice of noncrossing partitions associated with a finite Coxeter group.  Subsequently, other $M$-triangles were considered and computed for instance in \cites{armstrong09generalized,chapoton04enumerative,chapoton06sur,garver17enumerative,krattenthaler19rank,muehle18noncrossing}.  

An intriguing property of the $M$-triangle of noncrossing partition lattices is certain evaluations produce polynomials with nonnegative integer coefficients that combinatorially realize a refined counting of important objects in Coxeter--Catalan theory~\cite{armstrong09generalized}*{Section~5.3}.  See \cites{chapoton04enumerative,chapoton06sur} for the origins.  Translated to our setting, we are interested in the \defn{$F$-} and the \defn{$H$-triangle} associated with $\Hoch(n)$:
\begin{align*}
	F_{n}(x,y) & \defs y^{n}M_{n}\left(\frac{y+1}{y-x},\frac{y-x}{y}\right),\\
	H_{n}(x,y) & \defs \bigl(x(y-1)+1)^{n}M_{n}\left(\frac{y}{y-1},\frac{x(y-1)}{x(y-1)+1}\right).
\end{align*}

\begin{corollary}\label{cor:hochschild_fh_triangle}
	For $n>0$,
	\begin{align*}
		F_{n}(x,y) & = (x+y+1)^{n-2}\Bigl(nx^{2} + 2xy + (n+1)x + (y+1)^{2}\Bigr),\\
		H_{n}(x,y) & = (xy+1)^{n-2}\Bigl((xy+1)^{2} + (n-1)x\Bigr).
	\end{align*}
\end{corollary}
\begin{proof}
	This follows by definition from Theorem~\ref{thm:hochschild_clo_mtriangle}.
\end{proof}

\begin{example}\label{ex:hochschild_3_fh_triangle}
	Using the $M$-triangle computed in Example~\ref{ex:hochschild_3_m_triangle}, we notice that
	\begin{align*}
		F_{3}(x,y) & = (x+y+1)\Bigl(3x^{2} + 2xy + 4x + (y+1)^{2}\Bigr),\\
		H_{3}(x,y) & = (xy+1)\Bigl((xy+1)^{2} + 2x\Bigr).
	\end{align*}
\end{example}

\begin{remark}\label{rem:boolean_fh_triangle}
	If we define the analogous polynomials associated with the Boolean lattice, we obtain
	\begin{align*}
		F_{\Bool(n)}(x,y) & = (x+y+1)^{n},\\
		H_{\Bool(n)}(x,y) & = (xy+1)^{n}.
	\end{align*}
	Combinatorially, we may realize these polynomials as generating functions of intervals and elements in $\Bool(n)$, respectively:
	\begin{align*}
		F_{\Bool(n)} & = \sum_{A\subseteq B\subseteq[n]}x^{\lvert A\rvert}y^{n-\lvert A\rvert},\\
		H_{\Bool(n)} & = \sum_{A\subseteq[n]}(xy)^{\lvert A\rvert}.
	\end{align*}
\end{remark}
	
\subsection{Two combinatorial realizations of $F_{n}(x,y)$}
	\label{sec:hochschild_f_triangle}
We start by computing the coefficients of $F_{n}(x,y)$.

\begin{proposition}\label{prop:hochschild_f_triangle_coefficients}
	For $n>0$, the coefficient of $x^{k}y^{l}$ in $F_{n}(x,y)$ is
	\begin{displaymath}
		\binom{n}{k}\binom{n-k}{l}\left(\frac{n(k+1)-k(l+1)}{n}\right).
	\end{displaymath}
\end{proposition}
\begin{proof}
	From Corollary~\ref{cor:hochschild_fh_triangle}, we obtain
	\begin{align*}
		F_{n} & (x,y) = \Bigl(x+y+1\Bigr)^{n-2}\Bigl(nx^{2}+2xy+(n+1)x+(y+1)^{2}\Bigr)\\
		& = \left(\sum_{k=0}^{n-2}\binom{n-2}{k}x^{k}\sum_{l=0}^{n-2-k}\binom{n-2-k}{l}y^{l}\right)\\
			& \kern1cm \cdot\Bigl(nx^{2}+2xy+(n+1)x+(y+1)^{2}\Bigr)\\
		& = n\left(\sum_{k=0}^{n-2}\binom{n-2}{k}x^{k+2}\sum_{l=0}^{n-2-k}\binom{n-2-k}{l}y^{l}\right)\\
			& \kern1cm + 2\left(\sum_{k=0}^{n-2}\binom{n-2}{k}x^{k+1}\sum_{l=0}^{n-k}\binom{n-2-k}{l}y^{l+1}\right)\\
			& \kern1cm + (n+1)\left(\sum_{k=0}^{n-2}\binom{n-2}{k}x^{k+1}\sum_{l=0}^{n-2-k}\binom{n-2-k}{l}y^{l}\right)\\
			& \kern1cm + \left(\sum_{k=0}^{n-2}\binom{n-2}{k}x^{k}\sum_{l=0}^{n-k}\binom{n-k}{l}y^{l}\right)\\
		& = n\left(\sum_{k=0}^{n}\binom{n-2}{k-2}x^{k}\sum_{l=0}^{n-k}\binom{n-k}{l}y^{l}\right)\\
			& \kern1cm + 2\left(\sum_{k=0}^{n}\binom{n-2}{k-1}x^{k}\sum_{l=0}^{n-k}\binom{n-1-k}{l-1}y^{l}\right)\\
			& \kern1cm + (n+1)\left(\sum_{k=0}^{n}\binom{n-2}{k-1}x^{k}\sum_{l=0}^{n-k}\binom{n-1-k}{l}y^{l}\right)\\
			& \kern1cm + \left(\sum_{k=0}^{n}\binom{n-2}{k}x^{k}\sum_{l=0}^{n-k}\binom{n-k}{l}y^{l}\right).
	\end{align*}
	So, if $f_{n,k,l}$ denotes the coefficient of $x^{k}y^{l}$ in $F_{n}(x,y)$, then 
	\begin{align*}
		f_{n,k,l} & = n\binom{n-2}{k-2}\binom{n-k}{l} + 2\binom{n-2}{k-1}\binom{n-1-k}{l-1}\\
			& \kern2cm + (n+1)\binom{n-2}{k-1}\binom{n-1-k}{l} + \binom{n-2}{k}\binom{n-k}{l}\\
		& = \binom{n-2}{k-1}\binom{n-k}{l}\left(\frac{n(k-1)}{n-k}+\frac{2l}{n-k}+(n+1)\frac{n-k-l}{n-k}+\frac{n-k-1}{k}\right)\\
		& = \binom{n-2}{k-1}\binom{n-k}{l}\left(\frac{kl+n^2k-nkl+n^2-2nk-n+k}{(n-k)k}\right)\\
		& = \binom{n-1}{k}\binom{n-k}{l}\left(\frac{n(n-1)k-(n-1)kl-(n-1)k+n(n-1)}{(n-k)(n-1)}\right)\\
		& = \binom{n}{k}\binom{n-k}{l}\left(\frac{n(n-1)k-(n-1)kl-(n-1)k+n(n-1)}{n(n-1)}\right)\\
		& = \binom{n}{k}\binom{n-k}{l}\left(\frac{n(k+1)-k(l+1)}{n}\right).\qedhere
	\end{align*}
\end{proof}

Recall from Section~\ref{sec:hochschild_joinreps} that the canonical join representation of $\ufr\in\Tri(n)$ (as an element of $\Hoch(n)$) consists of join-irreducible triwords.  The join-irreducible elements of $\Hoch(n)$ are $\afr^{(i)}$ for $i\in[n]$ or $\bfr^{(i)}$ for $i\in\{2,3,\ldots,n\}$.  The atoms of $\Hoch(n)$ are those join-irreducible elements covering $\ofr$.  These comprise the following set
\begin{displaymath}
	\Atom(n) \defs \bigl\{\afr^{(1)},\bfr^{(2)},\bfr^{(3)},\ldots,\bfr^{(n)}\bigr\}.
\end{displaymath}
Since $\Atom((n)\subseteq\JI\bigl(\Hoch(n)\bigr)$, the canonical join representation of $\ufr\in\Tri(n)$ can be partitioned into atoms and non-atoms.  We use this property for combinatorially realizing the $F$- and the $H$-triangle.   For $\ufr\in\Tri(n)$, we define
\begin{align*}
	\neg(\ufr) & \defs \bigl\lvert\Can(\ufr)\cap\Atom(n)\bigr\rvert,
\end{align*}
and we consider the following polynomial:
\begin{displaymath}
	\tilde{F}_{n}(x,y) \defs \sum_{\ufr\in\Tri(n)}x^{n-\lvert\Can(\ufr)\rvert}(x+1)^{\lvert\Can(\ufr)\rvert-\neg(\ufr)}(y+1)^{\neg(\ufr)}.
\end{displaymath}

\begin{proposition}\label{prop:hochschild_f_triangle_combin}
	For $n>0$, it holds that $F_{n}(x,y)=\tilde{F}_{n}(x,y)$.
\end{proposition}
\begin{proof}
	Let $\tilde{f}_{n,k,l}$ denote the coefficient of $x^{k}y^{l}$ in $\tilde{F}_{n}(x,y)$, and pick $\ufr\in\Tri(n)$.  Suppose first that $\lvert\Can(\ufr)\rvert=n-k$.  If $l_{1}(\ufr)\leq 1$, then $\ufr$ contributes the term $x^{k}(y+1)^{n-k}$ to $\tilde{F}_{n}(x,y)$.  If $l_{1}(\ufr)>1$, then $\ufr$ contributes the term $x^{k}(x+1)(y+1)^{n-k-1}$.  Now suppose that $\lvert\Can(\ufr)\rvert=n-k+1$.  If $l_{1}(\ufr)\leq 1$, then $\ufr$ contributes the term $x^{k-1}(x+1)(y+1)^{n-k}$.  These are the only triwords contributing to the coefficients of a term involving $x^{k}$.  
	
	By Corollary~\ref{cor:hochschild_covers_clo_rank}, the size of the canonical join representation of $\ufr$ equals the rank of $\ufr$ in $\CLO\bigl(\Hoch(n)\bigr)$.  According to the proof of Proposition~\ref{prop:hochschild_clo_ranks}, the number of triwords $\ufr$ with $\rk(\ufr)=k$ and $l_{1}(\ufr)\leq 1$ is $\binom{n}{k}$.  The number of triwords with $\rk(\ufr)=k$ and $l_{1}(\ufr)>1$ is $(n-1)\binom{n-2}{k-1}$.  Thus, we obtain
	\begin{align*}
		\tilde{f}_{n,k,l} & = \binom{n}{n-k}\binom{n-k}{l} + (n-1)\binom{n-2}{n-k-1}\binom{n-k-1}{l}\\
			& \kern1cm + (n-1)\binom{n-2}{n-k}\binom{n-k}{l}\\
		& = \binom{n-k}{l}\left(\binom{n}{n-k}+\frac{(n-1)(n-k-l)}{n-k}\binom{n-2}{n-k-1}+(n-1)\binom{n-2}{n-k}\right)\\
		& = \left(1+\frac{(n-k-l)k}{n}+\frac{(k-1)k}{n}\right)\binom{n-k}{l}\binom{n}{k}\\
		& = \frac{n(k+1)-k(l+1)}{n}\binom{n}{k}\binom{n-k}{l}.
	\end{align*}
	Thus, by Proposition~\ref{prop:hochschild_f_triangle_coefficients}, $\tilde{f}_{n,k,l}$ is exactly the coefficient of $x^{k}y^{l}$ in $F_{n}(x,y)$, which establishes the claim.
\end{proof}

\begin{example}\label{ex:hochschild_3_f_triangle_1}
	By inspection of Figure~\ref{fig:hochschild_3}, we obtain the following values associated with the triwords of size $3$:
	
	\begin{center}\begin{tabular}{c||cccccc}
		$\ufr$ & $(0,0,0)$ & $(0,0,2)$ & $(0,2,0)$ & $(0,2,2)$ & $(1,0,0)$ & $(1,0,2)$\\ 
		\hline\hline
		$\lvert\Can(\ufr)\rvert$ & $0$ & $1$ & $1$ & $2$ & $1$ & $2$\\
		$\neg(\ufr)$ & $0$ & $1$ & $1$ & $2$ & $1$ & $2$\\
		\multicolumn{7}{c}{}\\
		$\ufr$ & $(1,1,0)$ & $(1,1,1)$ & $(1,1,2)$ & $(1,2,0)$ & $(1,2,1)$ & $(1,2,2)$\\
		\hline\hline
		$\lvert\Can(\ufr)\rvert$ & $1$ & $1$ & $2$ & $2$ & $2$ & $3$\\
		$\neg(\ufr)$ & $0$ & $0$ & $1$ & $2$ & $1$ & $3$\\
	\end{tabular}\end{center}
	We thus obtain
	\begin{align*}
		\tilde{F}_{3}(x,y) & = x^{3} + 3x^{2}(y+1) + 3x(y+1)^{2} + (y+1)^{3} + 2x^{2}(x+1) + 2x(x+1)(y+1)\\
		& = (x+y+1)\bigl(3x^{2}+2xy+4x+(y+1)^{2}\bigr)\\
		& = F_{3}(x,y).
	\end{align*}
\end{example}

Despite the fact that $\tilde{F}_{n}(x,y)$ combinatorially realizes $F_{n}(x,y)$, its nature is rather complicated, and its definition does not convey too much information as to what this polynomial essentially counts.  We now attempt a ``geometric'' explanation that is heavily inspired by the recent articles \cites{ceballos19sweak,ceballos21fh} and conversations with C.~Ceballos.  

Let $\Lattice=(L,\leq)$ be a finite lattice, and define by 
\begin{displaymath}
	\Cov_{\downarrow}(a) \defs \bigl\{a'\in L\mid (a',a)\in\Covers(\Lattice)\bigr\}
\end{displaymath}
the set of elements covered by $a$.  For $A\subseteq\Cov_{\downarrow}(a)$, we define the \defn{partial nucleus} of $a$ by
\begin{displaymath}
	a_{\downarrow A} \defs a\wedge\bigwedge_{a'\in A}a'.
\end{displaymath}
Moreover, the \defn{partial core} of $a$ is the interval $\Core_{A}(a)\defs[a_{\downarrow A},a]$.  

Note that, if $A=\emptyset$, then $a_{\downarrow A}=a$ and $\Core_{\emptyset}(a)=\{a\}$, and if $A=\Cov_{\downarrow}(a)$, then $a_{\downarrow A}$ is the nucleus of $a$ defined in Section~\ref{sec:core_label_order} and $\Core_{\Cov_{\downarrow}(a)}(a)$ is the core of $a$.

We now consider the set of all partial cores:
\begin{displaymath}
	\CP(\Lattice) \defs \bigl\{\Core_{A}(a)\mid a\in L,A\subseteq\Cov_{\downarrow}(a)\bigr\}.
\end{displaymath}

Applying this construction to a join-semidistributive lattice, we notice that $\Cov_{\downarrow}(a)$ essentially determines $\Can(a)$ via the map $\jsdlabeling$, see Theorem~\ref{thm:join_semidistributive_joinreps}.  

For the definition of $\tilde{F}_{n}(x,y)$, we have used a partition of the canonical join representation into atoms and non-atoms.  The reason for the shape of the resulting polynomial is better understood using $\CP\bigl(\Hoch(n)\bigr)$, if we define 
\begin{displaymath}
	\tilde{\neg}(\ufr,A) \defs \neg(\ufr)-\bigl\lvert A\cap\Atom(n)\bigr\rvert
\end{displaymath}
for $(\ufr,A)\in\CP\bigl(\Hoch(n)\bigr)$.

\begin{proposition}\label{prop:hochschild_cp_f_triangle_combin}
	For $n>0$,
	\begin{displaymath}
		F_{n}(x,y) = \sum_{(\ufr,A)\in\CP(\Hoch(n))}x^{n-\lvert A\rvert-\tilde{\neg}(\ufr,A)}y^{\tilde{\neg}(\ufr,A)}.
	\end{displaymath}
\end{proposition}
\begin{proof}
	This follows essentially from the Binomial Theorem.  Let us abbreviate 
	\begin{displaymath}
		\pos(\ufr) \defs \lvert\Can(\ufr)\setminus\Atom(n)\bigr\rvert = \lvert\Can(\ufr)\rvert-\neg(\ufr).
	\end{displaymath}
	Moreover, for $A\subseteq\Can(\ufr)$, we write $A_{+}=A\setminus\Atom(n)$ and $A_{-}=A\cap\Atom(n)$.  By Proposition~\ref{prop:hochschild_f_triangle_combin}, we have 
	\begin{align*}
		F_{n}(x,y) & = \sum_{\ufr\in\Tri(n)}x^{n-\lvert\Can(\ufr)\rvert}(x+1)^{\pos(\ufr)}(y+1)^{\neg(\ufr)}\\
		& = \sum_{\ufr\in\Tri(n)}x^{n-\lvert\Can(\ufr)\rvert}\sum_{i=0}^{\pos(\ufr)}\binom{\pos(\ufr)}{i}x^{\pos(\ufr)-i}\sum_{j=0}^{\neg(\ufr)}\binom{\neg(\ufr)}{j}y^{\neg(\ufr)-j}\\
		& = \sum_{\ufr\in\Tri(n)}x^{n-\lvert\Can(\ufr)\rvert}\sum_{A\subseteq\Cov_{\downarrow}(\ufr)}x^{\pos(\ufr)-\lvert A_{+}\rvert}y^{\neg(\ufr)-\lvert A_{-}\rvert}\\
		& = \sum_{(\ufr,A)\in\CP(\Hoch(n))}x^{n-\lvert A\rvert-(\neg(\ufr)-\lvert A_{-}\rvert)}y^{\neg(\ufr)-\lvert A_{-}\rvert}.\qedhere
	\end{align*}
\end{proof}

\begin{example}\label{ex:hochschild_3_f_triangle_2}
	Let us continue Example~\ref{ex:hochschild_3_f_triangle_1}.  If we consider $\ufr=(1,2,1)$.  Then 
	\begin{displaymath}
		\Cov_{\downarrow}(\ufr) = \bigl\{(1,2,0),(1,1,1)\bigr\},
	\end{displaymath}
	and $\neg(\ufr)=1$.  Let us write $\ufr_{1}=(1,2,0)$ and $\ufr_{2}=(1,1,1)$.  Then, $\ufr_{1}\notin\Atom(3)$ and $\ufr_{2}\in\Atom(3)$.  The partial cores associated with $\ufr$, together with the corresponding value of $\tilde{\neg}$ are
	
	\begin{center}\begin{tabular}{c||cccc}
		 $(\ufr,A)$ & $\bigl(\ufr,\emptyset\bigr)$ & $\bigl(\ufr,\{\ufr_{1}\}\bigr)$ & $\bigl(\ufr,\{\ufr_{2}\}\bigr)$ & $\bigl(\ufr,\{\ufr_{1},\ufr_{2}\}\bigr)$ \\
		 \hline\hline
		 $\tilde{\neg}(\ufr,A)$ & $1$ & $1$ & $0$ & $0$\\
	\end{tabular}\end{center}
	Therefore, the partial cores associated with $\ufr$ contribute the following terms to $F_{3}(x,y)$ in Proposition~\ref{prop:hochschild_cp_f_triangle_combin}
	\begin{displaymath}
		x^{2}y + xy + x^{2} + x = x(x+1)(y+1),
	\end{displaymath}
	which is precisely the term that $\ufr$ contributes to $F_{3}(x,y)$ (via $\tilde{F}_{3}(x,y)$) in Proposition~\ref{prop:hochschild_f_triangle_combin}.
\end{example}

Since $\Hoch(n)$ arises from the $n$-dimensional freehedron by acyclically orienting its $1$-skeleton, the nonempty faces of $\Free(n)$ are in bijection with the elements of $\CP\bigl(\Hoch(n)\bigr)$.  Indeed, this acyclic orientation equips every face $F$ of $\Free(n)$ with a unique source $a$ and a unique sink $b$.  If $B$ is the set of predecessors of $b$, then $a=b_{\downarrow B}$, and the vertices of $F$ correspond bijectively to the elements of $\Core_{B}(b)$.  

We use this connection to compute the face numbers of $\Free(n)$.  

\begin{proposition}\label{prop:hochschild_cp_faces}
	For $n>0$, the number of partial cores $(\ufr,A)\in\CP\bigl(\Hoch(n)\bigr)$ with $\lvert A\rvert=i$ is
	\begin{displaymath}
		2^{n-i-2}\binom{n}{i}\frac{n(n+3)-i(i-1)}{n}.
	\end{displaymath}
\end{proposition}
\begin{proof}
	Let $f_{i}$ denote the desired number, and let
	\begin{displaymath}
		f_{n}(x) = \sum_{i=0}^{n}f_{i}x^{i}.
	\end{displaymath}
	Then, by Proposition~\ref{prop:hochschild_cp_f_triangle_combin}, we have 
	\begin{align*}
		x^{n}F\left(\frac{1}{x},\frac{1}{x}\right) & = x^{n}\sum_{(\ufr,A)\in\CP(\Hoch(n))}\left(\frac{1}{x}\right)^{n-\lvert A\rvert-\tilde{\neg}(\ufr,A)}\left(\frac{1}{x}\right)^{\tilde{\neg}(\ufr,A)}\\
		& = \sum_{(\ufr,A)\in\CP(\Hoch(n))}x^{\lvert A\rvert}\\
		& = f_{n}(x).
	\end{align*}
	But this means precisely, that 
	\begin{displaymath}
		f_{i} = \sum_{\ufr\in\Tri(n)}\sum_{\substack{A\subseteq\Cov_{\downarrow}(\ufr),\\\lvert A\rvert=i}}1.
	\end{displaymath}
	By Corollary~\ref{cor:hochschild_covers_clo_rank}, for $\ufr\in\Tri(n)$, the cardinality of $\Cov_{\downarrow}(\ufr)$ equals the rank of $\ufr$ in $\CLO\bigl(\Hoch(n)\bigr)$.  Thus, by Proposition~\ref{prop:hochschild_clo_ranks}, we obtain
	\begin{align*}
		f_{i} & = \sum_{k=0}^{n}\binom{k}{i}\left(\binom{n}{k}+(n-k)\binom{n-1}{k-1}\right)\\
		& = \sum_{k=i}^{n}\binom{k}{i}\binom{n}{k} + \sum_{k=i}^{n}k\binom{k}{i}\binom{n-1}{k}\\
		& \overset{(*)}{=} \sum_{k=i}^{n}\binom{n}{i}\binom{n-i}{k-i} + \sum_{k=i}^{n}k\binom{n-1}{i}\binom{n-1-i}{k-i}\\
		& = \binom{n}{i}\sum_{k=0}^{n-i}\binom{n-i}{k} + \binom{n-1}{i}\sum_{k=0}^{n-i}(k+i)\binom{n-1-i}{k}\\
		& = \binom{n}{i}2^{n-i} + \binom{n-1}{i}\sum_{k=0}^{n-1-i}k\binom{n-1-i}{k}+i\binom{n-1}{i}\sum_{k=0}^{n-1-i}\binom{n-1-i}{k}\\
		& = \binom{n}{i}2^{n-i} + \binom{n-1}{i}(n-1-i)2^{n-2-i}+i\binom{n-1}{i}2^{n-i-1}\\
		& = \binom{n}{i}2^{n-i} + 2^{n-2-i}\binom{n-1}{i}\Bigl(n-1+i\Bigr)\\
		& = \binom{n}{i}2^{n-i} + 2^{n-2-i}\binom{n}{i}\frac{n^{2}-n+i-i^{2}}{n}\\
		& = \binom{n}{i}2^{n-i-2}\frac{n(n+3)-i(i-1)}{n},
	\end{align*}
	where $(*)$ follows from the ``trinomial revision'' $\binom{n}{k}\binom{k}{i}=\binom{n}{i}\binom{n-i}{k-i}$.
\end{proof}

\begin{corollary}\label{cor:freehedron_faces}
	For $n>0$, the number of faces of $\Free(n)$ of dimension $i$ is
	\begin{displaymath}
		2^{n-i-2}\binom{n}{i}\frac{n(n+3)-i(i-1)}{n}.
	\end{displaymath}
\end{corollary}

\subsection{Two combinatorial realizations of $H_{n}(x,y)$}
	\label{sec:hochschild_h_triangle}
By Corollary~\ref{cor:hochschild_fh_triangle}, we observe that
\begin{displaymath}
	H_{n}(x,1) = (x+1)^{n-2}\bigl(x^{2}+(n+1)x+1\bigr) = r_{n}(x).
\end{displaymath}
Therefore, we might expect that $H_{n}(x,y)$ can be realized using a refined rank-enumeration in $\CLO\bigl(\Hoch(n)\bigr)$.  By Corollary~\ref{cor:hochschild_covers_clo_rank}, the rank of $\ufr$ in $\CLO\bigl(\Hoch(n)\bigr)$ corresponds to the size of the canonical join representation of $\ufr$ in $\Hoch(n)$.  Using the partition of $\Can(\ufr)$ into atoms and non-atoms from the previous section suggests the following definition:
\begin{displaymath}
	\tilde{H}_{n}(x,y) \defs \sum_{\ufr\in\Tri(n)}x^{\lvert\Can(\ufr)\rvert}y^{\neg(\ufr)}.
\end{displaymath}

\begin{proposition}\label{prop:hochschild_h_triangle_combin}
	For $n>0$, it holds that $H_{n}(x,y) = \tilde{H}_{n}(x,y)$.
\end{proposition}
\begin{proof}
	By Corollary~\ref{cor:hochschild_fh_triangle}, we notice that the coefficient of $x^{k}y^{l}$ in $H_{n}(x,y)$ is
	\begin{displaymath}
		h_{n,k,l} = \begin{cases}
			\binom{n}{k}, & \text{if}\;k=l,\\
			(n-1)\binom{n-2}{k-1}, & \text{if}\;k=l+1,\\
			0, & \text{otherwise}.
		\end{cases}
	\end{displaymath}
	Now, let $\tilde{h}_{n,k,l}$ denote the coefficient of $x^{k}y^{l}$ in $\tilde{H}_{n}(x,y)$.  Let $\ufr\in\Tri(n)$ such that $\lvert\Can(\ufr)\rvert=k$.  If $l_{1}(\ufr)\leq 1$, then Proposition~\ref{prop:hochschild_canonical_joinrep} implies $\Can(\ufr)\subseteq\Atom(n)$.  Hence, $\ufr$ contributes to the coefficient $\tilde{h}_{n,k,k}$ and by Proposition~\ref{prop:hochschild_clo_ranks} there are $\binom{n}{k}$ such triwords.  If $l_{1}(\ufr)>1$, then $\bigl\lvert\Can(\ufr)\cap\Atom(n)\bigr\rvert=k-1$, and again by Proposition~\ref{prop:hochschild_clo_ranks} there are $(n-1)\binom{n-2}{k-1}$ ways for such a triword.  It follows that $\tilde{h}_{n,k,k-1}=h_{n,k,k-1}$.  Moreover, if $l\notin\{k{-}1,k\}$, then $\tilde{h}_{n,k,l}=0$.  
	
	We conclude that $h_{n,k,l}=\tilde{h}_{n,k,l}$ for all $k,l$ and thus $H_{n}(x,y)=\tilde{H}_{n}(x,y)$.
\end{proof}

\begin{example}\label{ex:hochschild_3_h_triangle_1}
	Using the values computed in Example~\ref{ex:hochschild_3_f_triangle_1}, we see that
	\begin{align*}
		\tilde{H}_{3}(x,y) & = 1 + 3xy + 3x^{2}y^{2} + x^{3}y^{3} + 2x + 2x^{2}y\\
		& = (xy+1)\bigl((xy+1)^{2}+2x\bigr)\\
		& = H_{3}(x,y).
	\end{align*}
\end{example}

The second realization of $H_{n}(x,y)$ is rather surprising.  The componentwise order on the join-irreducible triwords constitutes the disjoint union of an $n$-chain and an $(n{-}1)$-antichain.  

\begin{figure}
	\centering
	\begin{subfigure}[t]{.2\textwidth}
		\centering
		\begin{tikzpicture}\small
			\def\x{.8};
			\def\y{1};
			\draw(1*\x,1*\y) node(a1){$\afr^{(1)}$};
			\draw(1*\x,2*\y) node(a2){$\afr^{(2)}$};
			\draw(1*\x,3*\y) node(a3){$\afr^{(3)}$};
			\draw(1*\x,4*\y) node(a4){$\afr^{(4)}$};
			\draw(2*\x,1*\y) node(b2){$\bfr^{(2)}$};
			\draw(3*\x,1*\y) node(b3){$\bfr^{(3)}$};
			\draw(4*\x,1*\y) node(b4){$\bfr^{(4)}$};
			\draw(a1) -- (a2) -- (a3) -- (a4);
			\draw(b2) -- (a2);
		\end{tikzpicture}
		\caption{The poset $\Jb_{4}$.}
		\label{fig:extended_irreducible_poset_4}
	\end{subfigure}
	\hspace*{1cm}
	\begin{subfigure}[t]{.6\textwidth}
		\centering
		\begin{tikzpicture}\small
			\def\x{1.3};
			\def\y{1.5};
			\draw(1*\x,5*\y) node{\idls{}{$1$}{.5}};
			\draw(2*\x,5*\y) node{\idls{1/1}{$xy$}{.5}};
			\draw(3*\x,5*\y) node{\idls{2/1}{$xy$}{.5}};
			\draw(4*\x,5*\y) node{\idls{3/1}{$xy$}{.5}};
			\draw(5*\x,5*\y) node{\idls{4/1}{$xy$}{.5}};
			\draw(6*\x,5*\y) node{\idls{1/1,2/1}{$x^{2}y^{2}$}{.5}};
			\draw(1*\x,4*\y) node{\idls{1/1,3/1}{$x^{2}y^{2}$}{.5}};
			\draw(2*\x,4*\y) node{\idls{1/1,4/1}{$x^{2}y^{2}$}{.5}};
			\draw(3*\x,4*\y) node{\idls{2/1,3/1}{$x^{2}y^{2}$}{.5}};
			\draw(4*\x,4*\y) node{\idls{2/1,4/1}{$x^{2}y^{2}$}{.5}};
			\draw(5*\x,4*\y) node{\idls{3/1,4/1}{$x^{2}y^{2}$}{.5}};
			\draw(6*\x,4*\y) node{\idls{1/1,2/1,3/1}{$x^{3}y^{3}$}{.5}};
			\draw(1*\x,3*\y) node{\idls{1/1,2/1,4/1}{$x^{3}y^{3}$}{.5}};
			\draw(2*\x,3*\y) node{\idls{1/1,3/1,4/1}{$x^{3}y^{3}$}{.5}};
			\draw(3*\x,3*\y) node{\idls{2/1,3/1,4/1}{$x^{3}y^{3}$}{.5}};
			\draw(4*\x,3*\y) node{\idls{1/1,2/1,3/1,4/1}{$x^{4}y^{4}$}{.5}};
			\draw(5*\x,3*\y) node{\idls{1/2}{$x$}{.5}};
			\draw(6*\x,3*\y) node{\idls{1/2,3/1}{$x^{2}y$}{.5}};
			\draw(1*\x,2*\y) node{\idls{1/2,4/1}{$x^{2}y$}{.5}};
			\draw(2*\x,2*\y) node{\idls{1/2,3/1,4/1}{$x^{3}y^{2}$}{.5}};
			\draw(3*\x,2*\y) node{\idls{1/3}{$x$}{.5}};
			\draw(4*\x,2*\y) node{\idls{1/3,3/1}{$x^{2}y$}{.5}};
			\draw(5*\x,2*\y) node{\idls{1/3,4/1}{$x^{2}y$}{.5}};
			\draw(6*\x,2*\y) node{\idls{1/3,3/1,4/1}{$x^{3}y^{2}$}{.5}};
			\draw(1*\x,1*\y) node{\idls{1/4}{$x$}{.5}};
			\draw(2*\x,1*\y) node{\idls{1/4,3/1}{$x^{2}y$}{.5}};
			\draw(3*\x,1*\y) node{\idls{1/4,4/1}{$x^{2}y$}{.5}};
			\draw(4*\x,1*\y) node{\idls{1/4,3/1,4/1}{$x^{3}y^{2}$}{.5}};
		\end{tikzpicture}
		\caption{The antichains of $\Jb_{4}$ together with the term they contribute to $\tilde{H}_{\Jb_{4}}(x,y)$.  Minimal elements per antichain are marked in red.}
		\label{fig:antichains_4}
	\end{subfigure}
	\caption{Illustrating the combinatorial realization of $H_{\CLO\bigl(\Hoch(4)\bigr)}(x,y)$.}
	\label{fig:hochschild_h_triangle_combin_2}
\end{figure}
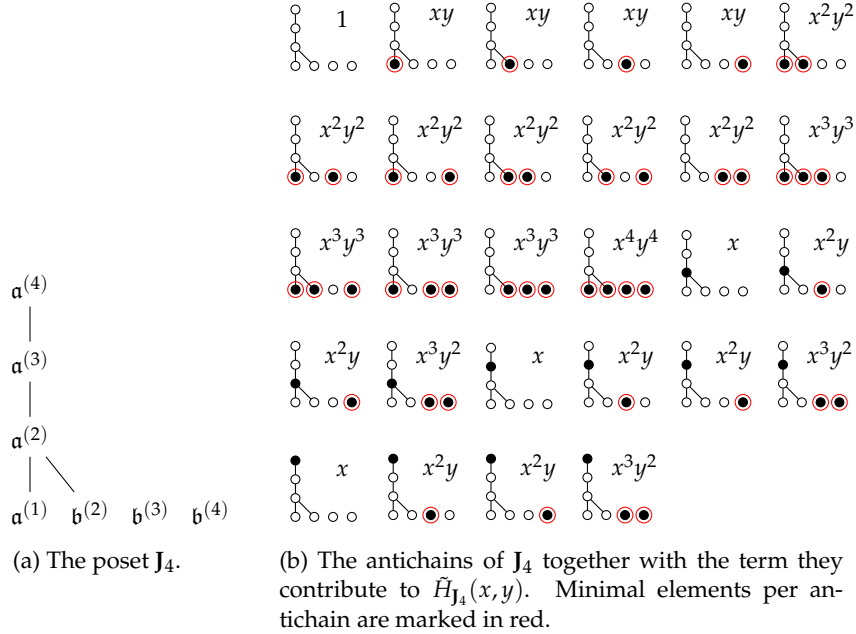

Let $\Jb_{n}$ denote the poset obtained from $\Bigl(\JI\bigl(\Hoch(n)\bigr),\comp\Bigr)$ by adding the relations $(\bfr^{(2)},\afr^{(i)})$ for $i>1$.  See Figure~\ref{fig:extended_irreducible_poset_4} for an illustration of $\Jb_{4}$.  Let $\AC(n)$ denote the set of antichains of $\Jb_{n}$.

\begin{proposition}\label{prop:hochschild_h_triangle_anti}
	For $n>0$,
	\begin{displaymath}
		H_{n}(x,y) = \sum_{A\in\AC(n)}x^{\lvert A\rvert}y^{\lvert A\cap\Atom(n)\rvert}.
	\end{displaymath}
\end{proposition}
\begin{proof}
	Let $A\in\AC(n)$ with $\lvert A\rvert=k$, and let $\bigl\lvert A\cap\Atom(n)\bigr\rvert=l$.  By the shape of $\Jb_{n}$ it is clear that $l\in\{k,k-1\}$.  If $l=k$, then there are $\binom{n}{k}$ possible choices for $A$.  If $l=k-1$, then $A$ contains neither $\afr^{(1)}$ nor $\bfr^{(2)}$, but has to contain $\afr^{(i)}$ for $i>1$.  Consequently, there are $(n-1)\binom{n-2}{k-1}$ possible choices for $A$.  As observed in the proof of Proposition~\ref{prop:hochschild_h_triangle_combin}, the number of such antichains equals the coefficient of $x^{k}y^{l}$ in $H_{n}(x,y)$.
\end{proof}

\begin{example}\label{ex:hochschild_3_h_triangle_2}
	Figure~\ref{fig:antichains_4} shows the antichains of $\Jb_{4}$, where the minimal elements per antichain are circled in red.  Additionally, we have noted the term each antichain contributes to $H_{4}(x,y)$.  We obtain
	\begin{align*}
		H_{4}(x,y) & = 1 + 4xy + 6x^{2}y^{2} + 4x^{3}y^{3} + x^{4}y^{4} + 3x + 6x^{2}y + 3x^{3}y^{2}\\
		& = (xy+1)^{2}\bigl((xy+1)^{2}+3x\bigr)
	\end{align*}
	as desired.
\end{example}

\begin{corollary}
	For $n>0$, the number of antichains of $\Jb_{n}$ is $2^{n-2}(n+3)$.
\end{corollary}
\begin{proof}
	This follows from Proposition~\ref{prop:triwords_size} by plugging in $x=y=1$ in Propositions~\ref{prop:hochschild_h_triangle_combin} and \ref{prop:hochschild_h_triangle_anti} and 
\end{proof}

\section{Open questions}
	\label{sec:open_questions}
\subsection{Shuffle lattices as core label orders}
	
By construction, we have $\ShufflePoset(n,0)\cong\Bool(n)$ and by \cite{muehle19the}*{Theorem~1.5}, $\CLO\bigl(\Bool(n)\bigr)\cong\Bool(n)$.  In Theorem~\ref{thm:hochschild_clo_shuffleposet} we have shown that $\CLO\bigl(\Hoch(n)\bigr)\cong\ShufflePoset(n-1,1)$.

Is there another family of semidistributive lattices, depending on parameters $n$ and $a$ whose core label orders realize $\ShufflePoset(n-a,a)$ for $a\geq 2$?

More precisely, the poset diagrams of both $\Bool(n)$ and $\Hoch(n)$ correspond to the (oriented) $1$-skeletons of the $n$-cube and the $n$-dimensional freehedron of \cites{rivera18combinatorial,saneblidze09bitwisted}, respectively.  Is there a family of polytopes or cell complexes, whose $1$-skeletons can be oriented such that one obtains extremal, congruence-uniform lattices whose core label orders realize $\ShufflePoset(n-a,a)$ for $a\geq 2$?

\subsection{Posets of join-irreducibles and $H$-triangles}
	\label{sec:h_triangle_irreducibles}
There is another family of lattices exhibiting a behavior similar to $\Hoch(n)$.  The \defn{Tamari lattice} $\Tamari(n)$ is a poset defined by a certain rotation transformation on the set of full binary trees with $n$ internal nodes~\cite{tamari51monoides}.  It was shown in \cites{geyer94on,markowsky92primes} that $\Tamari(n)$ is a congruence-uniform and extremal lattice, and its core label order is isomorphic to the lattice of noncrossing set partitions of $[n]$~\cite{reading11noncrossing}.  

The $M$-triangle associated with $\CLO\bigl(\Tamari(n)\bigr)$ was computed in \cite{athanasiadis07on}, and the corresponding $H$- and $F$-triangles were explained combinatorially in \cites{athanasiadis07on,chapoton06sur,thiel14on}.  Remarkably, the $H$-triangle can be realized analogously to Proposition~\ref{prop:hochschild_h_triangle_anti}, where antichains are taken in a triangular poset $\mathbf{T}_{n}$ with $\binom{n}{2}$ elements~\cite{chapoton06sur}.  The poset of join-irreducible elements of $\Tamari(n)$ is isomorphic to the disjoint union of $n-1$ chains of lengths $1,2,\ldots,n-1$, respectively~\cite{bennett94two}.  The triangular poset $\mathbf{T}_{n}$ is clearly an order extension of $\JIPoset\bigl(\Tamari(n)\bigr)$.  

Figure~\ref{fig:polytopes_lattices} illustrates this connection on the Boolean lattice, the Hochschild lattice and the Tamari lattice.

Can we find other families of semidistributive lattices $\bigl\{\Lattice_{n}\mid n\in\mathbb{N}\bigr\}$, such that the $H$-triangle, arising from the $M$-triangle of $\CLO(\Lattice)$, can be realized via a refined antichain enumeration in some order extension $\Poset_{n}$ of the poset of join-irreducibles of $\Lattice_{n}$ such that $\bigl\lvert\AC(\Poset_{n})\bigr\rvert=\bigl\lvert L_{n}\bigr\rvert$?

\begin{landscape}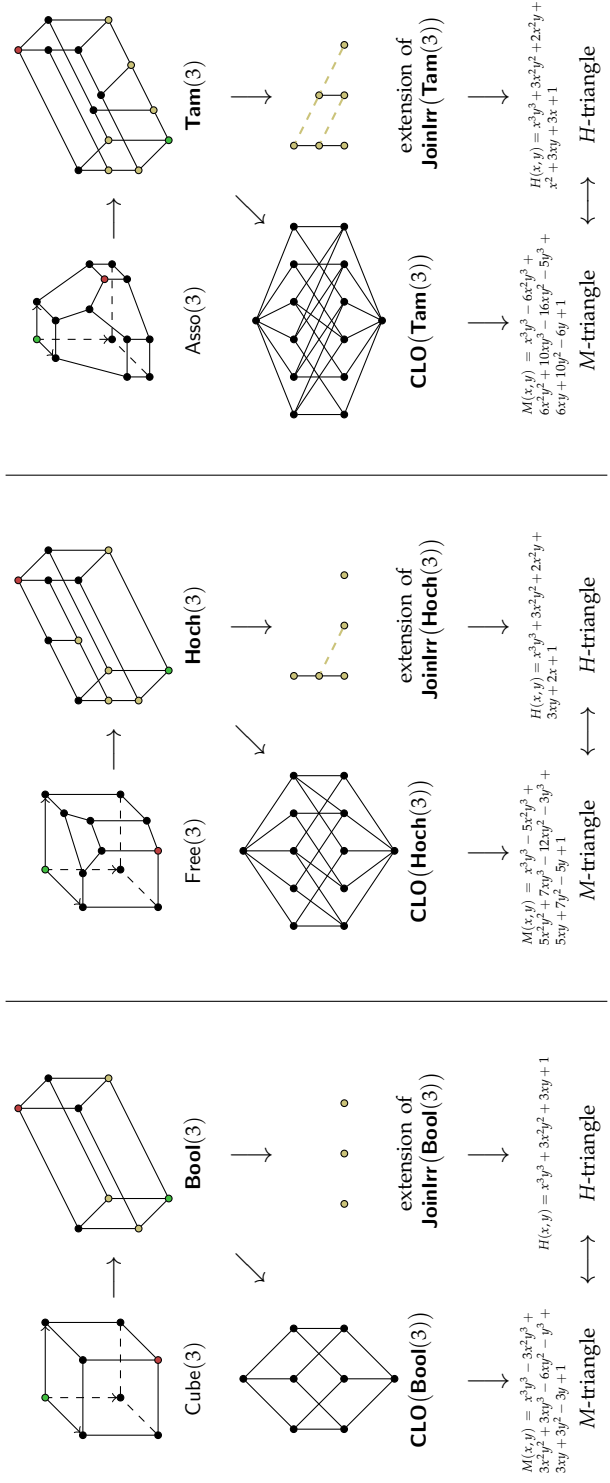
\begin{figure}[p]
	\centering
	\begin{tikzpicture}
		\draw(5,8) node{};
		\draw(4.5,-3) -- (4.5,5);
		\draw(11.5,-3) -- (11.5,5);
		\draw(1,1) node{\begin{tikzpicture}
				\def\dx{3};
				\def\dy{3};
				\def\s{.3};
				\draw(1*\dx,3*\dy) node{\begin{tikzpicture}\small
					\def\x{1};
					\def\y{1};
					\draw(1*\x,1*\y) node[fnode](n1){};
					\draw(2*\x,1*\y) node[fnode,fill=red!50!gray](n2){};
					\draw(1.5*\x,1.5*\y) node[fnode](n3){};
					\draw(2.5*\x,1.5*\y) node[fnode](n4){};
					\draw(1*\x,2*\y) node[fnode](n5){};
					\draw(2*\x,2*\y) node[fnode](n6){};
					\draw(1.5*\x,2.5*\y) node[fnode,fill=green!50!gray](n7){};
					\draw(2.5*\x,2.5*\y) node[fnode](n8){};
					\draw(2*\x,2.67*\y) node{};
					\draw(n1) -- (n2);
					\draw[dashed](n1) -- (n3);
					\draw(n1) -- (n5);
					\draw(n2) -- (n4);
					\draw(n2) -- (n6);
					\draw[dashed](n3) -- (n4);
					\draw[<-,dashed](n3) -- (n7);
					\draw(n4) -- (n8);
					\draw(n5) -- (n6);
					\draw[<-](n5) -- (n7);
					\draw(n6) -- (n8);
					\draw[->](n7) -- (n8);
				\end{tikzpicture}};
				\draw(1*\dx,2.55*\dy) node[scale=.8]{$\mathsf{Cube}(3)$};
				\draw(2*\dx,3*\dy) node{\begin{tikzpicture}\small
					\def\x{.4};
					\def\y{.4};
					\draw(3*\x,1*\y) node[fnode,fill=green!50!gray](n1){};
					\draw(2*\x,2*\y) node[fnode,fill=yellow!50!gray](n2){};
					\draw(3*\x,3*\y) node[fnode,fill=yellow!50!gray](n3){};
					\draw(7*\x,3*\y) node[fnode,fill=yellow!50!gray](n4){};
					\draw(2*\x,4*\y) node[fnode](n5){};
					\draw(6*\x,4*\y) node[fnode](n6){};
					\draw(7*\x,5*\y) node[fnode](n7){};
					\draw(6*\x,6*\y) node[fnode,fill=red!50!gray](n8){};
					\draw(n1) -- (n2);
					\draw(n1) -- (n3);
					\draw(n1) -- (n4);
					\draw(n2) -- (n5);
					\draw(n2) -- (n6);
					\draw(n3) -- (n5);
					\draw(n3) -- (n7);
					\draw(n4) -- (n6);
					\draw(n4) -- (n7);
					\draw(n5) -- (n8);
					\draw(n6) -- (n8);
					\draw(n7) -- (n8);
				\end{tikzpicture}};
				\draw(2*\dx,2.55*\dy) node[scale=.8]{$\Bool(3)$};
				\draw(1*\dx,2*\dy) node{\begin{tikzpicture}\small
					\def\x{.67};
					\def\y{.67};
					\draw(2*\x,1*\y) node[fnode](n1){};
					\draw(1*\x,2*\y) node[fnode](n2){};
					\draw(2*\x,2*\y) node[fnode](n3){};
					\draw(3*\x,2*\y) node[fnode](n4){};
					\draw(1*\x,3*\y) node[fnode](n5){};
					\draw(2*\x,3*\y) node[fnode](n6){};
					\draw(3*\x,3*\y) node[fnode](n7){};
					\draw(2*\x,4*\y) node[fnode](n8){};
					\draw(n1) -- (n2);
					\draw(n1) -- (n3);
					\draw(n1) -- (n4);
					\draw(n2) -- (n5);
					\draw(n2) -- (n6);
					\draw(n3) -- (n5);
					\draw(n3) -- (n7);
					\draw(n4) -- (n6);
					\draw(n4) -- (n7);
					\draw(n5) -- (n8);
					\draw(n6) -- (n8);
					\draw(n7) -- (n8);
				\end{tikzpicture}};
				\draw(1*\dx,1.55*\dy) node[scale=.8]{$\CLO\bigl(\Bool(3)\bigr)$};
				\draw(2*\dx,2*\dy) node{\begin{tikzpicture}\small
					\def\x{.67};
					\def\y{.67};
					\draw(1*\x,1*\y) node{};
					\draw(1*\x,2*\y) node[fnode,fill=yellow!50!gray](n1){};
					\draw(2*\x,2*\y) node[fnode,fill=yellow!50!gray](n2){};
					\draw(3*\x,2*\y) node[fnode,fill=yellow!50!gray](n3){};
					\draw(3*\x,4*\y) node{};
				\end{tikzpicture}};
				\draw(2*\dx,1.55*\dy) node[scale=.8,text width=2.5cm,text centered]{extension of $\JIPoset\bigl(\Bool(3)\bigr)$};
				\draw(1*\dx,1*\dy) node[text width=5cm,scale=.5]{$M(x,y) = x^{3}y^{3}-3x^{2}y^{3}+3x^{2}y^{2}+3xy^{3}-6xy^{2}-y^{3}+3xy+3y^{2}-3y+1$};
				\draw(1*\dx,.8*\dy) node[scale=.8]{$M$-triangle};
				\draw(2*\dx,1*\dy) node[text width=5.2cm,scale=.5]{$H(x,y) = x^{3}y^{3}+3x^{2}y^{2}+3xy+1$};
				\draw(2*\dx,.8*\dy) node[scale=.8]{$H$-triangle};
				\draw(1.53*\dy,.8*\dy) node{$\longleftrightarrow$};
				\draw(1*\dx,1.25*\dy) node[rotate=90]{$\longleftarrow$};
				\draw(2*\dx,1.25*\dy) node[rotate=90]{$\longleftarrow$};
				\draw(2*\dy,2.3*\dy) node[rotate=90]{$\longleftarrow$};
				\draw(1.5*\dy,2.3*\dy) node[rotate=45]{$\longleftarrow$};
				\draw(1.45*\dy,2.9*\dy) node{$\longrightarrow$};
			\end{tikzpicture}};
		
		\draw(8,1) node{\begin{tikzpicture}
				\def\dx{3};
				\def\dy{3};
				\def\s{.3};
				\draw(1*\dx,3*\dy) node{\begin{tikzpicture}\small
					\def\x{1};
					\def\y{1};
					\draw(1*\x,1*\y) node[fnode](n1){};
					\draw(1.75*\x,1*\y) node[fnode,fill=red!50!gray](n2){};
					\draw(1.5*\x,1.5*\y) node[fnode](n3){};
					\draw(2.15*\x,1.15*\y) node[fnode](n4){};
					\draw(2.5*\x,1.5*\y) node[fnode](n5){};
					\draw(1.75*\x,1.75*\y) node[fnode](n6){};
					\draw(2.15*\x,1.9*\y) node[fnode](n7){};
					\draw(1*\x,2*\y) node[fnode](n8){};
					\draw(1.45*\x,2*\y) node[fnode](n9){};
					\draw(2.25*\x,2.25*\y) node[fnode](n10){};
					\draw(1.5*\x,2.5*\y) node[fnode,fill=green!50!gray](n11){};
					\draw(2.5*\x,2.5*\y) node[fnode](n12){};
					\draw(2*\x,2.67*\y) node{};
					\draw(n1) -- (n2);
					\draw[dashed](n1) -- (n3);
					\draw(n1) -- (n8);
					\draw(n2) -- (n4);
					\draw(n2) -- (n6);
					\draw[dashed](n3) -- (n5);
					\draw[<-,dashed](n3) -- (n11);
					\draw(n4) -- (n5);
					\draw(n4) -- (n7);
					\draw(n5) -- (n12);
					\draw(n6) -- (n7);
					\draw(n7) -- (n10);
					\draw(n8) -- (n9);
					\draw[<-](n8) -- (n11);
					\draw(n9) -- (n6);
					\draw(n9) -- (n10);
					\draw(n10) -- (n12);
					\draw[->](n11) -- (n12);
				\end{tikzpicture}};
				\draw(1*\dx,2.55*\dy) node[scale=.8]{$\mathsf{Free}(3)$};
				\draw(2*\dx,3*\dy) node{\begin{tikzpicture}\small
					\def\x{.4};
					\def\y{.4};
					\draw(3*\x,1*\y) node[fnode,fill=green!50!gray](n1){};
					\draw(2*\x,2*\y) node[fnode,fill=yellow!50!gray](n2){};
					\draw(2*\x,3*\y) node[fnode,fill=yellow!50!gray](n3){};
					\draw(3*\x,3*\y) node[fnode,fill=yellow!50!gray](n4){};
					\draw(7*\x,3*\y) node[fnode,fill=yellow!50!gray](n5){};
					\draw(2*\x,4*\y) node[fnode](n6){};
					\draw(4*\x,4*\y) node[fnode,fill=yellow!50!gray](n7){};
					\draw(6*\x,4*\y) node[fnode](n8){};
					\draw(4*\x,5*\y) node[fnode](n9){};
					\draw(6*\x,5*\y) node[fnode](n10){};
					\draw(7*\x,5*\y) node[fnode](n11){};
					\draw(6*\x,6*\y) node[fnode,fill=red!50!gray](n12){};
					\draw(n1) -- (n2);
					\draw(n1) -- (n4);
					\draw(n1) -- (n5);
					\draw(n2) -- (n3);
					\draw(n2) -- (n8);
					\draw(n3) -- (n6);
					\draw(n3) -- (n7);
					\draw(n4) -- (n6);
					\draw(n4) -- (n11);
					\draw(n5) -- (n8);
					\draw(n5) -- (n11);
					\draw(n6) -- (n9);
					\draw(n7) -- (n9);
					\draw(n7) -- (n10);
					\draw(n8) -- (n10);
					\draw(n9) -- (n12);
					\draw(n10) -- (n12);
					\draw(n11) -- (n12);
				\end{tikzpicture}};
				\draw(2*\dx,2.55*\dy) node[scale=.8]{$\Hoch(3)$};
				\draw(1*\dx,2*\dy) node{\begin{tikzpicture}\small
					\def\x{.5};
					\def\y{.67};
					\draw(3*\x,1*\y) node[fnode](n1){};
					\draw(1*\x,2*\y) node[fnode](n2){};
					\draw(2*\x,2*\y) node[fnode](n3){};
					\draw(3*\x,2*\y) node[fnode](n4){};
					\draw(4*\x,2*\y) node[fnode](n5){};
					\draw(5*\x,2*\y) node[fnode](n6){};
					\draw(1*\x,3*\y) node[fnode](n7){};
					\draw(2*\x,3*\y) node[fnode](n8){};
					\draw(3*\x,3*\y) node[fnode](n9){};
					\draw(4*\x,3*\y) node[fnode](n10){};
					\draw(5*\x,3*\y) node[fnode](n11){};
					\draw(3*\x,4*\y) node[fnode](n12){};
					\draw(n1) -- (n2);
					\draw(n1) -- (n3);
					\draw(n1) -- (n4);
					\draw(n1) -- (n5);
					\draw(n1) -- (n6);
					\draw(n2) -- (n8);
					\draw(n2) -- (n7);
					\draw(n2) -- (n9);
					\draw(n3) -- (n9);
					\draw(n3) -- (n10);
					\draw(n4) -- (n7);
					\draw(n4) -- (n11);
					\draw(n5) -- (n8);
					\draw(n5) -- (n10);
					\draw(n5) -- (n11);
					\draw(n6) -- (n9);
					\draw(n6) -- (n11);
					\draw(n7) -- (n12);
					\draw(n8) -- (n12);
					\draw(n9) -- (n12);
					\draw(n10) -- (n12);
					\draw(n11) -- (n12);
				\end{tikzpicture}};
				\draw(1*\dx,1.55*\dy) node[scale=.8]{$\CLO\bigl(\Hoch(3)\bigr)$};
				\draw(2*\dx,2*\dy) node{\begin{tikzpicture}\small
					\def\x{.67};
					\def\y{.67};
					\draw(1*\x,1*\y) node{};
					\draw(1*\x,2*\y) node[fnode,fill=yellow!50!gray](n1){};
					\draw(2*\x,2*\y) node[fnode,fill=yellow!50!gray](n2){};
					\draw(3*\x,2*\y) node[fnode,fill=yellow!50!gray](n3){};
					\draw(1*\x,2.5*\y) node[fnode,fill=yellow!50!gray](n4){};
					\draw(1*\x,3*\y) node[fnode,fill=yellow!50!gray](n5){};
					\draw(3*\x,4*\y) node{};
					\draw(n1) -- (n4);
					\draw(n4) -- (n5);
					\draw[thick,dashed,yellow!50!gray](n2) -- (n4);
				\end{tikzpicture}};
				\draw(2*\dx,1.55*\dy) node[scale=.8,text width=2.5cm,text centered]{extension of $\JIPoset\bigl(\Hoch(3)\bigr)$};
				\draw(1*\dx,1*\dy) node[text width=5cm,scale=.5]{$M(x,y) = x^{3}y^{3}-5x^{2}y^{3}+5x^{2}y^{2}+7xy^{3}-12xy^{2}-3y^{3}+5xy+7y^{2}-5y+1$};
				\draw(1*\dx,.8*\dy) node[scale=.8]{$M$-triangle};
				\draw(2*\dx,1*\dy) node[text width=5cm,scale=.5]{$H(x,y) = x^{3}y^{3}+3x^{2}y^{2}+2x^{2}y+3xy+2x+1$};
				\draw(2*\dx,.8*\dy) node[scale=.8]{$H$-triangle};
				\draw(1.53*\dy,.8*\dy) node{$\longleftrightarrow$};
				\draw(1*\dx,1.25*\dy) node[rotate=90]{$\longleftarrow$};
				\draw(2*\dx,1.25*\dy) node[rotate=90]{$\longleftarrow$};
				\draw(2*\dy,2.3*\dy) node[rotate=90]{$\longleftarrow$};
				\draw(1.5*\dy,2.3*\dy) node[rotate=45]{$\longleftarrow$};
				\draw(1.45*\dy,2.9*\dy) node{$\longrightarrow$};
			\end{tikzpicture}};
			
		\draw(15,1) node{\begin{tikzpicture}
				\def\dx{3};
				\def\dy{3};
				\def\s{.3};
				\draw(1*\dx,3*\dy) node{\begin{tikzpicture}\small
					\def\x{1};
					\def\y{1};
					\draw(1*\x,1*\y) node[fnode](n1){};
					\draw(1.5*\x,1*\y) node[fnode](n2){};
					\draw(1*\x,1.3*\y) node[fnode](n3){};
					\draw(1.5*\x,1.3*\y) node[fnode](n4){};
					\draw(2.3*\x,1.3*\y) node[fnode](n5){};
					\draw(1.5*\x,1.5*\y) node[fnode](n6){};
					\draw(2.5*\x,1.5*\y) node[fnode](n7){};
					\draw(2.3*\x,1.6*\y) node[fnode,fill=red!50!gray](n8){};
					\draw(1.9*\x,1.8*\y) node[fnode](n9){};
					\draw(2.5*\x,1.8*\y) node[fnode](n10){};
					\draw(1.25*\x,2.25*\y) node[fnode](n11){};
					\draw(1.75*\x,2.25*\y) node[fnode](n12){};
					\draw(1.5*\x,2.5*\y) node[fnode,fill=green!50!gray](n13){};
					\draw(2*\x,2.5*\y) node[fnode](n14){};
					\draw(n1) -- (n2);
					\draw(n1) -- (n3);
					\draw(n2) -- (n4);
					\draw(n2) -- (n5);
					\draw(n3) -- (n4);
					\draw(n3) -- (n11);
					\draw(n4) -- (n9);
					\draw(n5) -- (n7);
					\draw(n5) -- (n8);
					\draw(n7) -- (n10);
					\draw(n8) -- (n9);
					\draw(n8) -- (n10);
					\draw(n9) -- (n12);
					\draw(n10) -- (n14);
					\draw(n11) -- (n12);
					\draw[<-](n11) -- (n13);
					\draw(n12) -- (n14);
					\draw[->](n13) -- (n14);
					\draw[dashed](n1) -- (n6);
					\draw[dashed](n6) -- (n7);
					\draw[<-,dashed](n6) -- (n13);
				\end{tikzpicture}};
				\draw(1*\dx,2.55*\dy) node[scale=.8]{$\mathsf{Asso}(3)$};
				\draw(2*\dx,3*\dy) node{\begin{tikzpicture}\small
					\def\x{.4};
					\def\y{.4};
					\draw(3*\x,1*\y) node[fnode,fill=green!50!gray](n1){};
					\draw(4*\x,1.5*\y) node[fnode,fill=yellow!50!gray](n2){};
					\draw(2*\x,2*\y) node[fnode,fill=yellow!50!gray](n3){};
					\draw(5.5*\x,2.25*\y) node[fnode,fill=yellow!50!gray](n4){};
					\draw(2*\x,3*\y) node[fnode,fill=yellow!50!gray](n5){};
					\draw(3*\x,3*\y) node[fnode,fill=yellow!50!gray](n6){};
					\draw(7*\x,3*\y) node[fnode,fill=yellow!50!gray](n7){};
					\draw(4.5*\x,3.25*\y) node[fnode](n8){};
					\draw(4*\x,3.5*\y) node[fnode](n9){};
					\draw(2*\x,4*\y) node[fnode](n10){};
					\draw(6*\x,4*\y) node[fnode](n11){};
					\draw(6*\x,5*\y) node[fnode](n12){};
					\draw(7*\x,5*\y) node[fnode](n13){};
					\draw(6*\x,6*\y) node[fnode,fill=red!50!gray](n14){};
					\draw(n1) -- (n2);
					\draw(n1) -- (n3);
					\draw(n1) -- (n6);
					\draw(n2) -- (n4);
					\draw(n2) -- (n9);
					\draw(n3) -- (n5);
					\draw(n3) -- (n8);
					\draw(n4) -- (n7);
					\draw(n4) -- (n8);
					\draw(n5) -- (n10);
					\draw(n5) -- (n12);
					\draw(n6) -- (n9);
					\draw(n6) -- (n10);
					\draw(n7) -- (n11);
					\draw(n7) -- (n13);
					\draw(n8) -- (n11);
					\draw(n9) -- (n13);
					\draw(n10) -- (n14);
					\draw(n13) -- (n14);
					\draw(n11) -- (n12);
					\draw(n12) -- (n14);
				\end{tikzpicture}};
				\draw(2*\dx,2.55*\dy) node[scale=.8]{$\Tamari(3)$};
				\draw(1*\dx,2*\dy) node{\begin{tikzpicture}\small
					\def\x{.5};
					\def\y{.67};
					\draw(3.5*\x,1.25*\y) node[fnode](n1){};
					\draw(1*\x,2*\y) node[fnode](n2){};
					\draw(2*\x,2*\y) node[fnode](n3){};
					\draw(3*\x,2*\y) node[fnode](n4){};
					\draw(4*\x,2*\y) node[fnode](n5){};
					\draw(5*\x,2*\y) node[fnode](n6){};
					\draw(6*\x,2*\y) node[fnode](n7){};
					\draw(1*\x,3*\y) node[fnode](n8){};
					\draw(2*\x,3*\y) node[fnode](n9){};
					\draw(3*\x,3*\y) node[fnode](n10){};
					\draw(4*\x,3*\y) node[fnode](n11){};
					\draw(5*\x,3*\y) node[fnode](n12){};
					\draw(6*\x,3*\y) node[fnode](n13){};
					\draw(3.5*\x,3.75*\y) node[fnode](n14){};
					\draw(n1) -- (n2);
					\draw(n1) -- (n3);
					\draw(n1) -- (n4);
					\draw(n1) -- (n5);
					\draw(n1) -- (n6);
					\draw(n1) -- (n7);
					\draw(n2) -- (n8);
					\draw(n2) -- (n10);
					\draw(n3) -- (n9);
					\draw(n3) -- (n11);
					\draw(n4) -- (n8);
					\draw(n4) -- (n11);
					\draw(n4) -- (n12);
					\draw(n5) -- (n8);
					\draw(n5) -- (n9);
					\draw(n5) -- (n13);
					\draw(n6) -- (n9);
					\draw(n6) -- (n10);
					\draw(n6) -- (n12);
					\draw(n7) -- (n10);
					\draw(n7) -- (n11);
					\draw(n7) -- (n13);
					\draw(n8) -- (n14);
					\draw(n9) -- (n14);
					\draw(n10) -- (n14);
					\draw(n11) -- (n14);
					\draw(n12) -- (n14);
					\draw(n13) -- (n14);
				\end{tikzpicture}};
				\draw(1*\dx,1.55*\dy) node[scale=.8]{$\CLO\bigl(\Tamari(3)\bigr)$};
				\draw(2*\dx,2*\dy) node{\begin{tikzpicture}\small
					\def\x{.67};
					\def\y{.67};
					\draw(1*\x,1*\y) node{};
					\draw(1*\x,2*\y) node[fnode,fill=yellow!50!gray](n1){};
					\draw(2*\x,2*\y) node[fnode,fill=yellow!50!gray](n2){};
					\draw(3*\x,2*\y) node[fnode,fill=yellow!50!gray](n3){};
					\draw(1*\x,2.5*\y) node[fnode,fill=yellow!50!gray](n4){};
					\draw(2*\x,2.5*\y) node[fnode,fill=yellow!50!gray](n5){};
					\draw(1*\x,3*\y) node[fnode,fill=yellow!50!gray](n6){};
					\draw(3*\x,4*\y) node{};
					\draw(n1) -- (n4);
					\draw(n4) -- (n6);
					\draw(n2) -- (n5);
					\draw[thick,dashed,yellow!50!gray](n2) -- (n4);
					\draw[thick,dashed,yellow!50!gray](n3) -- (n5);
					\draw[thick,dashed,yellow!50!gray](n5) -- (n6);
				\end{tikzpicture}};
				\draw(2*\dx,1.55*\dy) node[scale=.8,text width=2.5cm,text centered]{extension of $\JIPoset\bigl(\Tamari(3)\bigr)$};
				\draw(1*\dx,1*\dy) node[text width=5cm,scale=.5]{$M(x,y) = x^{3}y^{3}-6x^{2}y^{3}+6x^{2}y^{2}+10xy^{3}-16xy^{2}-5y^{3}+6xy+10y^{2}-6y+1$};
				\draw(1*\dx,.8*\dy) node[scale=.8]{$M$-triangle};
				\draw(2*\dx,1*\dy) node[text width=5cm,scale=.5]{$H(x,y) = x^{3}y^{3}+3x^{2}y^{2}+2x^{2}y+x^{2}+3xy+3x+1$};
				\draw(2*\dx,.8*\dy) node[scale=.8]{$H$-triangle};
				\draw(1.53*\dy,.8*\dy) node{$\longleftrightarrow$};
				\draw(1*\dx,1.25*\dy) node[rotate=90]{$\longleftarrow$};
				\draw(2*\dx,1.25*\dy) node[rotate=90]{$\longleftarrow$};
				\draw(2*\dy,2.3*\dy) node[rotate=90]{$\longleftarrow$};
				\draw(1.5*\dy,2.3*\dy) node[rotate=45]{$\longleftarrow$};
				\draw(1.45*\dy,2.9*\dy) node{$\longrightarrow$};
			\end{tikzpicture}};
	\end{tikzpicture}
	\caption{Three polytopes with associated lattices.}
	\label{fig:polytopes_lattices}
\end{figure}\end{landscape}

\subsection{Interval enumeration in shuffle posets}

For $a,b\geq 0$ we consider the polynomial
\begin{displaymath}
	G_{a,b}(x,y) \defs \sum_{\substack{\ab,\bb\in\Shuffle(a,b)\\\ab\preceq\bb}}x^{\rk(\ab)}y^{a+b-\rk(\bb)}.
\end{displaymath}
If $a=n$ and $b=0$, then $\ShufflePoset(a,b)\cong\Bool(n)$, and by Remark~\ref{rem:boolean_fh_triangle}, we have
\begin{displaymath}
	G_{n,0}(x,y) = (x+y+1)^{n} = F_{\Bool(n)}(x,y).
\end{displaymath}
For $a<n$ and $b>0$, the $G$-triangle does no longer coincide with the $F$-triangle.  We conjecture the following explicit formula for the case $a=n-1$ and $b=1$, which can be verified for $a=2,b=1$ in Figure~\ref{fig:shuffle_21}.

\begin{conjecture}
	For $n>0$ we have
	\begin{displaymath}
		G_{n-1,1}(x,y) = (x+y+1)^{n-2}\Bigl(x^{2}+y^{2}+1+(n+1)(xy+x+y)\Bigr).
	\end{displaymath}
\end{conjecture}

\subsection{The geometric structure of partial cores}

By construction, the Hochschild lattice $\Hoch(n)$ arises as an orientation of the $1$-skeleton of the freehedron $\Free(n)$.  Therefore, the nonempty faces of $\Free(n)$ correspond bijectively to the partial cores of $\Hoch(n)$.  Can we equip the set $\CP\bigl(\Hoch(n)\bigr)$ with an ``intersection'' operation such that $\CP\bigl(\Hoch(n)\bigr)\cup\{\emptyset\}$ is combinatorially isomorphic to $\Free(n)$?  

More generally, given a finite lattice $\Lattice$, under what conditions is $\CP(\Lattice)\cup\{\emptyset\}$ a cell complex?

What is the connection between the core label order of a semidistributive lattice $\Lattice$ and the containment order on $\CP(\Lattice)$, determined by containment of intervals?

\section*{Acknowledgements}
	\label{sec:acknowledgements}
I want to thank Camille Combe for interesting discussions on the Hochschild lattice and cubical lattices.

\end{document}